\documentclass[10pt]{amsart}

\usepackage[english]{babel}
\usepackage[utf8]{inputenc}
\usepackage{microtype}

\usepackage{amsmath, amsfonts, amssymb, amsthm, mathtools}
\usepackage{stmaryrd}

\usepackage{enumitem}

\usepackage{xcolor}
\definecolor{utahRed}{RGB}{190, 0, 0}

\usepackage{graphicx}
\graphicspath{{./images/}}

\usepackage{subcaption}
\usepackage{float}
\usepackage{thmtools}

\usepackage{hyperref}
\usepackage[capitalise, noabbrev, nameinlink]{cleveref}
\hypersetup{
    colorlinks = true,
    linkcolor= utahRed,
    citecolor= utahRed,
    urlcolor = utahRed,
    breaklinks = true,
}

\newtheorem{theorem}{Theorem}[section]

\newtheorem{lemma}[theorem]{Lemma}
\newtheorem{proposition}[theorem]{Proposition}
\newtheorem{thmx}{Theorem}

\numberwithin{equation}{section}

\theoremstyle{definition}
\newtheorem{definition}[theorem]{Definition}
\newtheorem{notation}[theorem]{Notation}

\newtheorem{example}[theorem]{Example}

\newtheorem{question}[theorem]{Question}
\newtheorem*{ack}{Acknowledgments}

\newcommand{\mc}{\mathcal}

\newcommand{\mingens}{\operatorname{mingens}}
\newcommand{\lcm}{\operatorname{lcm}}
\newcommand{\starnbd}{\operatorname{star}}
\newcommand{\cone}{\operatorname{cone}}
\newcommand{\Scarf}{\operatorname{Scarf}}
\newcommand{\taylor}{\operatorname{Taylor}}

\title{Scarf complexes of connected and path ideals}

\author[Chau]{Trung Chau}
\address{Chennai Mathematical Institute, Siruseri, Tamil Nadu 603103. India}
\email{chauchitrung1996@gmail.com}

\author[Sheng]{Richie Sheng}
\address{Department of Mathematics, University of Utah, Salt Lake City, UT, USA}
\email{u1415944@utah.edu}

\author[Tribone]{Tim Tribone}
\address{Department of Mathematics, University of Utah, Salt Lake City, UT, USA}
\email{tim.tribone@utah.edu}
\urladdr{\href{https://timtribone.com/}{timtribone.com}}

\author[Wooton]{Deborah Wooton}
\address{Department of Mathematics, University of Utah, Salt Lake City, UT, USA}
\email{deborah.wooton@utah.edu}

\subjclass{13F55, 13D02}
\keywords{Scarf complex, monomial ideal, connected ideal, path ideal, free resolution}

\begin{document}

\begin{abstract}
    The $t$-connected ideal of a graph $G$ is generated by all connected induced subgraphs of $G$ with $t$ vertices. When $t=2$, this coincides with the usual edge ideal of the graph. Following the work of Faridi et al., we give a classification of the graphs whose $t$-connected ideals are minimally resolved by their Scarf complex. We also consider the $t$-path ideal of a graph $G$ which is the ideal generated by all paths of length $t$ in $G$. In this case, we are able to give a classification of the same type for paths of length $t=4$.
\end{abstract}

\maketitle

\section{Introduction}

    Let $S = \Bbbk[x_1,\dots,x_d]$ be a polynomial ring over a field $\Bbbk$ and $I$ a monomial ideal of $S$. This work addresses the classical problem in commutative algebra of understanding the minimal free resolution of the $S$-module $S/I$. Following the work of many others, we will consider the special case of monomial ideals arising from finite simple graphs \cite{MR4245096,Chau19082025,CK24,MR2457403,MR2875826,faridi2024scarf}.

    The approach we take is possible due to a result of Bayer-Peeva-Sturmfels \cite{bayer1998monomial}, which we recall below. The guiding theme of this approach is to associate to the monomial ideal $I$ a simplicial complex whose homology detects homological properties of $I$. In our setting, the graph defining $I$ and the homology of the associated simplicial complex both contribute to the understanding of the minimal free resolution of $S/I$. Indeed, while our motivation comes from commutative algebra, our results rely primarily on graph-theoretic characterizations made in \cref{sec:forbidden_subgraphs} and basic properties of simplicial homology.

\subsection*{Labeled simplicial complexes}

    Suppose $I$ is minimally generated by monomials $\mingens(I) \coloneqq \{m_1,m_2,\dots,m_q\}$, and let $\Delta$ be a simplicial complex on $\mingens(I)$, i.e. $\Delta$ is a collection of subsets of $\mingens(I)$ which is closed under taking subsets (and contains the empty set). As a convention, we also assume that $\Delta$ contains the singleton sets consisting of each element of $\mingens(I)$. We equip $\Delta$ with the following ``labeling": given a \textit{face} $\sigma \in \Delta$ (which is a subset of $\mingens(I)$ contained in $\Delta$), we label $\sigma$ with the least common multiple of its elements, which we denote by $\lcm(\sigma)$. For our purposes, we will set $\lcm \varnothing = 1$.
    
    From the data of a labeled simplicial complex, one can construct a complex of free $S$-modules denoted $\mc F_\Delta$; this process is called \emph{homogenization}, and we refer the reader to \cite{bayer1998monomial} for the explicit description of $\mc F_\Delta$. The simplicial complex $\Delta$ is said to \textit{support a resolution of} $I$ if the complex $\mc F_\Delta$ is acyclic, that is, if the complex $\mc F_\Delta$, augmented by $S/I$, is a free resolution of $S/I$.
    
    The criterion given by Bayer-Peeva-Sturmfels is as follows: the complex $\Delta$ supports a resolution of $I$ if and only if, for any monomial $m$ in $S$, the subcomplex
        \[\Delta_{\leq m} \coloneqq \bigl\{\sigma \in \Delta : \lcm(\sigma) \big\vert m \bigr\} \subseteq \Delta\]
    is either empty or has trivial reduced homology with coefficients in $\Bbbk$; see \cite[Lemma 2.2]{bayer1998monomial} or \cref{lem:BPS} below. In other words, to check whether $\mc F_\Delta$ is a resolution, we can check  whether a finite number of subcomplexes of $\Delta$ are acyclic over $\Bbbk$; this is often easier than directly computing the homology of $\mc F_\Delta$. 

    The simplest example of a simplicial complex associated to $I$ is the \textit{Taylor complex}, which we denote by $\taylor(I)$. This is the simplicial complex which contains all subsets of $\mingens(I)$, i.e., the power set of $\mingens(I)$. Geometrically, this can be thought of as the full $(q-1)$-dimensional simplex on the vertices $m_1,\dots,m_q$. For any monomial $m$, the subcomplex $\taylor(I)_{\leq m}$ is also a simplex (or is empty) and thus $\mc F_{\taylor(I)}$ is a resolution; this was originally shown by Taylor in \cite{Tay66}, hence the name ``Taylor complex". Since $\taylor(I)$ always supports a resolution of $I$, the resolution $\mc F_{\taylor(I)}$ is commonly referred to as the \textit{Taylor resolution} of $S/I$.
    
    It is worth noting that the Taylor resolution is rarely minimal. Indeed, it is guaranteed to be non-minimal as the number of generators of $I$ grows large compared to the number of variables. One approach to addressing this minimality issue is to consider a subcomplex of $\taylor(I)$ called the Scarf complex, which was introduced by Bayer-Peeva-Sturmfels in \cite{bayer1998monomial}.

\subsection*{The Scarf complex} 

    As above, let $I$ be a monomial ideal of $S$. The \textit{Scarf complex} of $I$, denoted $\Scarf(I)$, is the subset of $\taylor(I)$ obtained by removing all faces with non-unique lcm labels. In other words,
        \[\Scarf(I) \coloneqq \{\sigma \in \taylor(I) : \lcm(\sigma) = \lcm(\tau) \text{ implies } \sigma = \tau\}.\]
    
    The benefits of the Scarf complex are apparent as one starts working with them. For instance, for any monomial ideal $I$,
        \begin{itemize}
            \item $\Scarf(I)$ is a simplicial subcomplex of $\taylor(I)$;
            \item $\mc F_{\Scarf(I)}$ is a minimal complex;
            \item $\mc F_{\Scarf(I)}$ is contained in the minimal free resolution of $S/I$.
        \end{itemize}
    Thus, whenever $\Scarf(I)$ supports a resolution of $I$, $\mc F_{\Scarf(I)}$ is necessarily the minimal free resolution of $S/I$.
    
    On the other hand, there is one clear drawback to the Scarf complex: removing non-unique labels may produce homology; that is, the complex $\mc F_{\Scarf(I)}$ is not always a resolution of $S/I$. This leads us to the main question addressed in this article.

    \begin{question}\label{main_question}
        Which monomial ideals $I$ are minimally resolved by the Scarf complex $\mc F_{\Scarf(I)}$?
    \end{question}
    
    Somewhat surprisingly, \textit{most} monomial ideals admit their Scarf complex as their minimal free resolution, emphasizing the ubiquity of the complex; see \cite{bayer1998monomial} for the precise notion of ``most''. However, it is largely unknown exactly which square-free monomial ideals have this property. 
    
    In the case when $I$ is the edge ideal (or a power of an edge ideal) of a graph, a complete answer to \cref{main_question} was given in \cite{faridi2024scarf}. Their ``Beautiful Oberwolfach Theorem'' shows that the graphs whose edge ideals are resolved by the Scarf complex are forests with small diameter; see \cite[Theorem 8.3]{faridi2024scarf} for the full statement or \cref{sec:connected_ideals} for the case of a connected graph. Continuing in this direction, the authors in \cite{Chau19082025} classify the square-free and symbolic powers of edge ideals and cover ideals that admit a Scarf minimal resolution.

    This article addresses \cref{main_question} in the case of a $t$-connected or $t$-path ideal associated to a graph, which we define below. We consider these ideals for $t > 2$, since the $2$-connected and $2$-path ideals of a graph both coincide with the edge ideal. 

    \begin{definition}\label{def:path_conn_intro}
        Let $t\geq 2$ be an integer and $G$ a simple connected graph on vertices $\{x_1,x_2,\dots,x_d\}$.
            \begin{enumerate}[label = (\roman*)] 
                \item \label{def:conn_intro} The \textit{$t$-connected ideal} of $G$, denoted $\mc C_t(G)$, is the square-free monomial ideal of $S$ generated by all products of the form $x_{i_1}x_{i_2}\cdots x_{i_t}$ such that $i_j \neq i_k$ for all $j \neq k$ and the induced subgraph of $G$ generated by $\{x_{i_1},\dots,x_{i_t}\}\subseteq V(G)$ is connected.
    
                \item \label{def:path_intro} The \textit{$t$-path ideal} of $G$, denoted $\mc P_t(G)$, is the square-free monomial ideal of $S$ generated by all products of the form $x_{i_1}x_{i_2}\cdots x_{i_t}$ such that $i_j \neq i_k$ for all $j\neq k$ and the vertices $x_{i_1},x_{i_2},\dots,x_{i_t}$ form a path in $G$.
            \end{enumerate}
    \end{definition}

    Path ideals were first considered by Conca and De Negri \cite{ConcaDeNegri1998} to study determinantal and Pfaffian ideals. They are a popular class of ideals among researchers in commutative algebra (see, e.g., \cite{AlilooeeFaridi2015,AlilooeeFaridi2018,Banerjee2017,Erey2020,KianiMadani2016}), partly because they are a generalization of the classic edge ideals of graphs. Connected ideals are another generalization of edge ideals, but have received less attention compared to path ideals. They arise from the independent complex of a graph, and thus have applications in algebraic topology (see, e.g., \cite{DRSV24,AJM25}). Our first main result, which is an extension of \cite[Theorem 8.3]{faridi2024scarf}, classifies the graphs whose $t$-connected ideals are minimally resolved by their Scarf complex.

    \begin{thmx}[\cref{thm:connected-scarf-iff-short-path}]\label{thm:A} 
        Let $G$ be a connected graph and $t \geq 3$ an integer. Then the ideal $\mc C_t(G)$ is minimally resolved by its Scarf complex if and only if $G$ has at most $t$ vertices, or $G$ is a path on $r$ vertices for some $r \leq 2t$.    
    \end{thmx}

    We also consider the $t$-path ideals of a graph. In this case, the situation is significantly more complicated. The ideals $\mc C_3(G)$ and $\mc P_3(G)$ coincide, and thus the case $t=4$ is the next unknown case. We are able to give the following classification.

    \begin{thmx}[\cref{thm:P_4_Scarf}]\label{thm:B}
        Let $G$ be a connected graph. Then the ideal $\mc P_4(G)$ is minimally resolved by its Scarf complex if and only if $G$ has at most four vertices or $G$ is isomorphic to one of the graphs listed in \cref{fig:S_graphs}:
        \begin{itemize}
            \item $T_k$ for $k \ge 0$;
            \item $S_k$ for $k \ge 0$;
            \item $S_k^{m,n}$ for $k=3,4$ and $m,n \ge 0$;
            \item $S_k^{m,n,p}$ for $k=5,6$ and $m,n,p \ge 0$.
        \end{itemize}
    \end{thmx}

    \begin{figure}[ht]\label{fig:P_4_scarf_graphs}
        \centering

        \begin{subfigure}[b]{0.155\textwidth}
            \centering
            \includegraphics[width=\textwidth]{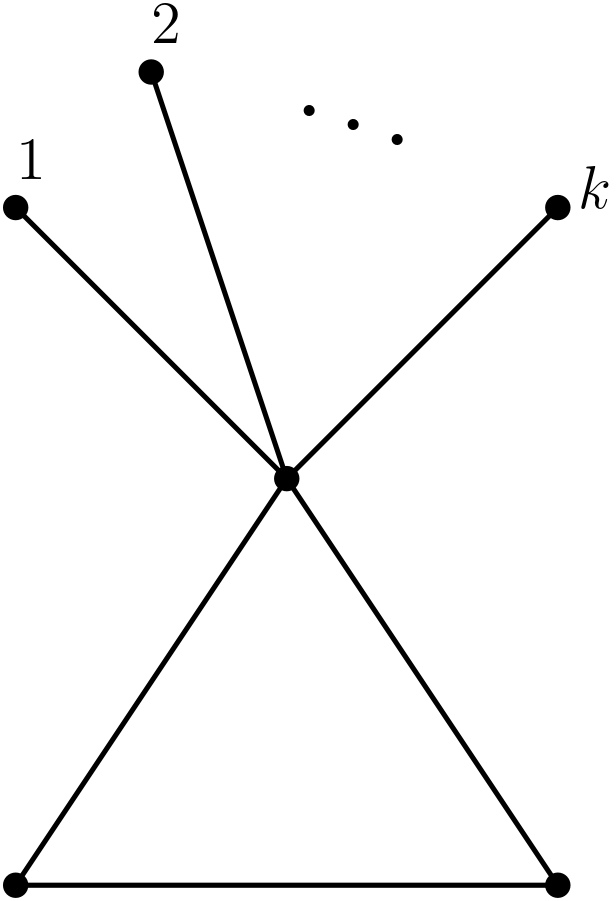}
            \caption{\(T_k\)}
        \end{subfigure}
        \hfill
        \begin{subfigure}[b]{0.155\textwidth}
            \centering
            \includegraphics[width=\textwidth]{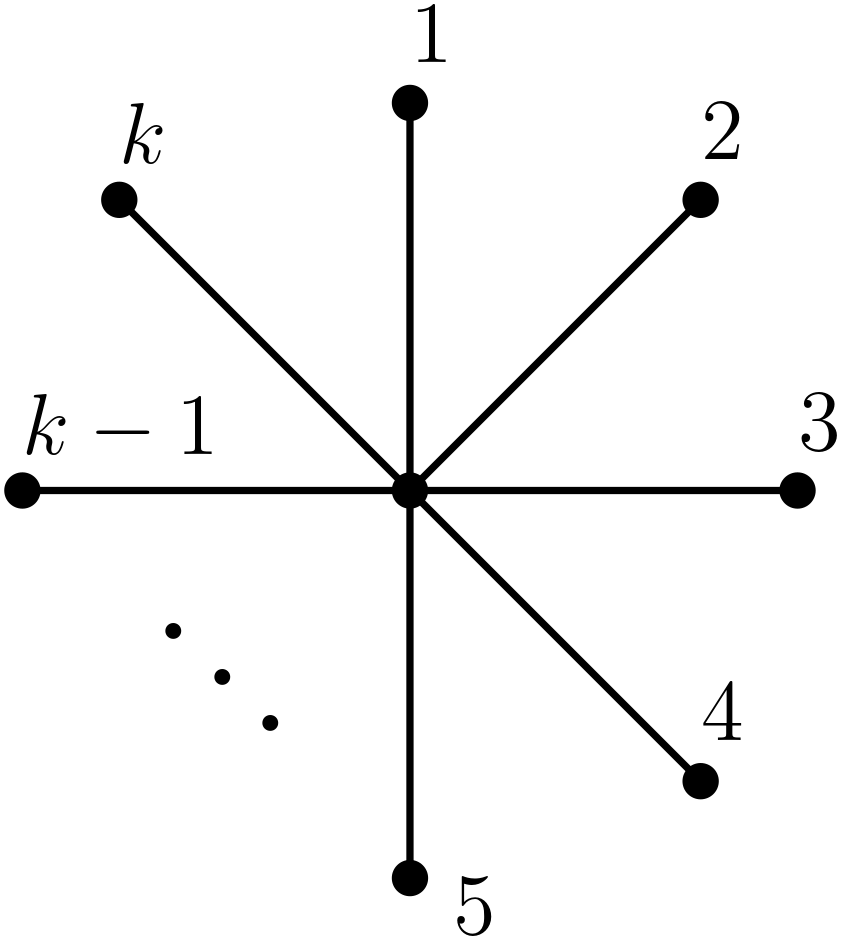}
            \caption{\(S_k\)}
        \end{subfigure}
        \hfill
        \begin{subfigure}[b]{0.155\textwidth}
            \centering
            \includegraphics[width=\textwidth]{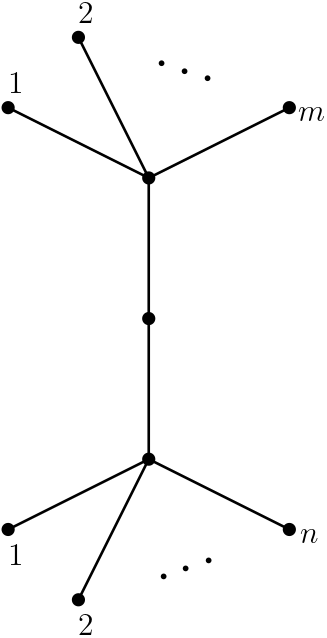}
            \caption{\(S_3^{m, n}\)}
        \end{subfigure}
        \hfill
        \begin{subfigure}[b]{0.155\textwidth}
            \centering
            \includegraphics[width=\textwidth]{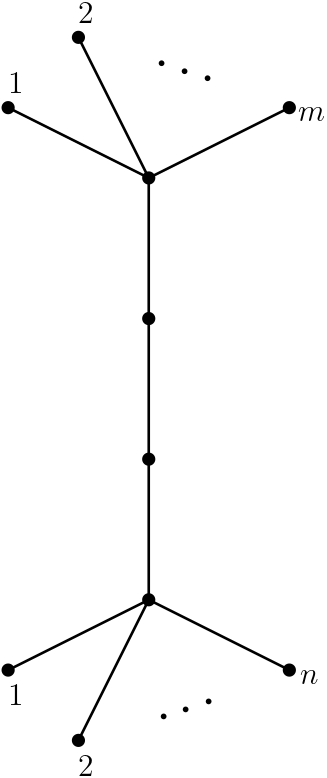}
            \caption{\(S_4^{m, n}\)}
        \end{subfigure}
        \hfill
        \begin{subfigure}[b]{0.155\textwidth}
            \centering
            \includegraphics[width=\textwidth]{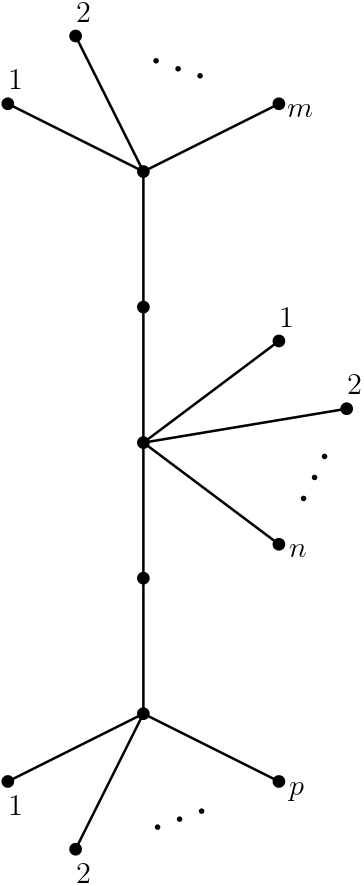}
            \caption{\(S_5^{m, n, p}\)}
        \end{subfigure}
        \hfill
        \begin{subfigure}[b]{0.155\textwidth}
            \centering
            \includegraphics[width=\textwidth]{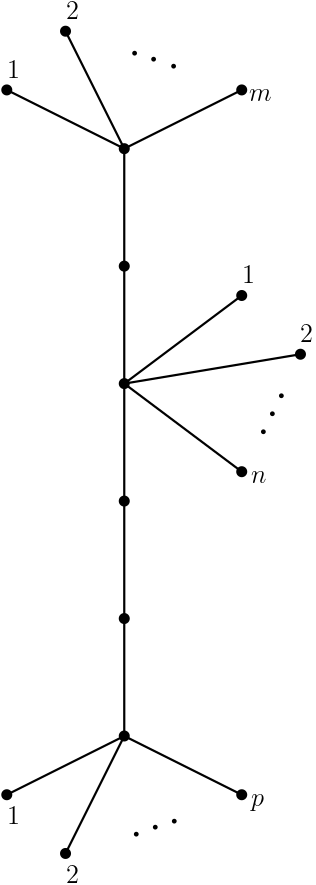}
            \caption{\(S_6^{m, n, p}\)}
        \end{subfigure}
        
        \caption{Graphs which are $\mc P_4$-Scarf.}
        \label{fig:S_graphs}
    \end{figure}

    The paper is structured as follows. In \cref{sec:setup}, we recall the construction of monomial resolutions from simplicial complexes, specifically Scarf complexes, and related results. \cref{sec:forbidden_subgraphs} focuses on the characterization of several classes of graphs based on their forbidden induced subgraphs, which are important in establishing our main results. In \cref{sec:connected_ideals}, we study the Scarf complex of connected ideals, with the main result being \cref{thm:A} (\cref{thm:connected-scarf-iff-short-path}). Finally, a study of 4-path ideals, with \cref{thm:B} (\cref{thm:P_4_Scarf}) as the main goal, is covered in \cref{sec:p4_scarf}.

    \begin{ack}
        The first author acknowledges support from the Infosys Foundation. The second and fourth authors were supported by the University of Utah Mathematics Department REU and the NSF RTG Grant $\#1840190$.  We would also like to thank Nathan Uricchio for his helpful comments and suggestions.
    \end{ack}

\section{Monomial resolutions from simplicial complexes}\label{sec:setup}

    In this section, we collect the definitions and notation which will be used throughout. Many of our results rely on \cref{lem:BPS} below. For this reason, our presentation focuses mainly on the relevant simplicial complexes and their properties. We also record several key ``restriction lemmas'' in the sense of \cite{HHZ04}.

    Let $S = \Bbbk[x_1,\dots, x_d]$ be a polynomial ring over a field $\Bbbk$ and $I \subseteq S$ a monomial ideal which is minimally generated by monomials $\mingens(I) = \{m_1,\dots, m_q\}$. Given a simplicial complex $\Delta$ on $\mingens(I)$, we let $\mc F_\Delta$ denote the complex of free $S$-modules defined in \cite[Section 2]{bayer1998monomial}. When this complex is acyclic, and therefore forms a free resolution of $S/I$, we say that $\Delta$ \textit{supports a resolution} of $S/I$. In the case that $\Delta = \Scarf(I)$ supports a resolution of $S/I$, we will say that $I$ is \textit{Scarf}. By construction, the complex $\mc F_{\Scarf(I)}$ is the minimal free resolution of $S/I$ when the ideal $I$ is Scarf.
    
    Determining whether or not a given simplicial complex supports a resolution of $I$ often relies on the following key result given by Bayer-Peeva-Sturmfels.    

    \begin{lemma}[\protect{\cite[Lemma 2.2]{bayer1998monomial}}]\label{lem:BPS}
         Let $I \subseteq S$ be a monomial ideal and $\Delta$ a simplicial complex with vertices labeled with elements of $\mingens(I)$. Then the complex $\mathcal{F}_\Delta$ is a resolution of $S/I$ if and only if for any monomial $m$, the simplicial complex $\Delta_{\leq m} \coloneqq \bigl\{ \sigma  \in \Delta : \lcm(\sigma) \big\vert m \bigr\} $ is either empty or acyclic over $\Bbbk$.
    \end{lemma}
    
    In this way, the homology of $\Delta$ (and its subcomplexes) is reflected in the homology of the chain complex $\mc F_\Delta$, and vice versa. In particular, notice that $\Delta$ itself must be acyclic over $\Bbbk$ in order to support a resolution of $I$.

    \begin{example}
        Let $G$ denote a path on six vertices $x_1,x_2,x_3,x_4,x_5,x_6$. The $3$-connected ideal of $G$ is generated by four monomials:
            \[\mc C_3(G) = (x_1x_2x_3, x_2x_3x_4, x_3x_4x_5, x_4x_5x_6).\]
        For each $1 \leq i \leq 4$, the singleton sets $\{x_ix_{i+1}x_{i+2}\}$ and the two element subsets $\{x_ix_{i+1}x_{i+2},x_{i+1}x_{i+2}x_{i+3}\}$ have unique lcms. However, the lcm of the subset $\{x_1x_2x_3,x_4x_5x_6\}$, which is the product of all of the variables, is not unique. It follows that $\Scarf \bigl(\mc C_3(G) \bigr)$ itself a path and we will see in \cref{sec:connected_ideals} that $\Scarf \bigl(\mc C_3(G) \bigr)$ supports a resolution of $\mc C_3(G)$.
        
        \begin{figure}[h]
            \centering
            \includegraphics[width=0.5\textwidth]{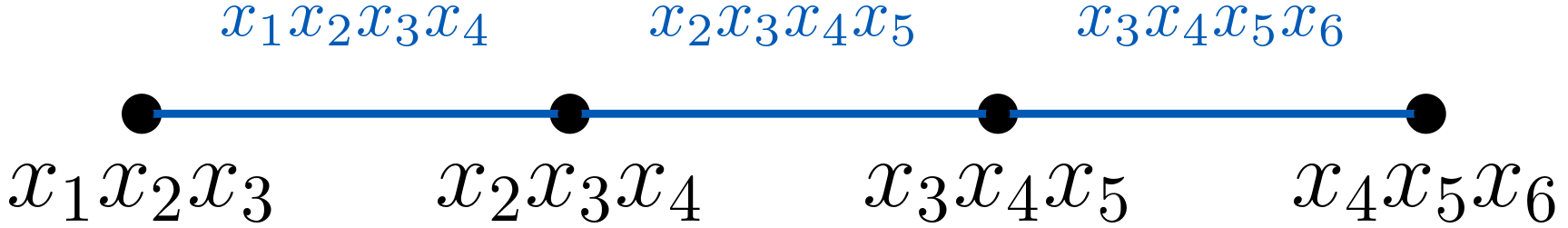}
            \caption{\(\Scarf \bigl(\mc C_3(G) \bigr)\). The label on a face is its lcm.}
        \end{figure}

        On the other hand, if $G'$ is a path on seven vertices, using the same reasoning as above, $\mc C_3(G')$ is generated by five monomials and $\Scarf \bigl(\mc C_3(G') \bigr)$ is the boundary of a pentagon. Hence, $\Scarf \bigl(\mc C_3(G') \bigr)$ has non-trivial first homology and therefore does not support a resolution of $\mc C_3(G')$.
    \end{example}

\subsection{Restriction Lemmas}

    Let $G = \bigl(V(G),E(G)\bigr)$ be a graph where $V(G)$ is the vertex set of $G$ and $E(G)$ is the edge set of $G$. We will consider non-empty finite simple graphs, that is, non-empty graphs with finitely many vertices and whose edge sets contain no repeated edges or loops. Our main focus will be on connected graphs and their connected (induced) subgraphs. A \emph{subgraph} $H$ of $G$ will refer to a collection $H = \bigl(V(H), E(H)\bigr)$ such that $V(H) \subseteq V(G)$ and $E(H) \subseteq E(G)$ with the stipulation that each edge in $E(H)$ has endpoints in $V(H)$. An \textit{induced} subgraph of $G$ is a subgraph $H$ whose edge set consists of all possible edges in $E(G)$ between vertices of $H$. In particular, if $V' \subseteq V(G)$, the induced subgraph of $G$ \textit{generated} by $V'$ consists of the vertices $V'$ and all edges in $E(G)$ which have both endpoints in $V'$. 

    We will make repeated use of a version of the Restriction Lemma \cite[Lemma 4.4]{HHZ04}. If $I$ is a monomial ideal and $m$ a monomial, we let $I^{\leq m}$ denote the monomial ideal generated by the elements of $\mingens(I)$ that divide $m$.

    \begin{lemma}[\protect{\cite[Lemma 2.4]{chau2025realizing}}]
        If $I$ is Scarf, then so is $I^{\leq m}$ for any monomial $m$.
    \end{lemma}

    Now let $G$ be a graph, $H$ an induced subgraph of $G$, and $m$ the product of all vertices in $H$. By definition, the ideals $\mc C_t(G)^{\leq m}$ and $\mc P_t(G)^{\leq m}$ are exactly $\mc C_t(H)$ and $\mc P_t(H)$, respectively. Thus, we obtain the following analogue of \cite[Lemma~4.4]{HHZ04}.

    \begin{lemma}\label{lem:Restriction-Scarf} 
        If $\mc C_t(G)$ (resp, $\mc P_t(G)$) is Scarf, then so is $\mc C_t(H)$ (resp, $\mc P_t(H)$) for any induced subgraph $H$ of $G$. In other words, if $\mc C_t(H)$ (resp, $\mc P_t(H)$) is not Scarf for some graph $H$, then neither is $\mc C_t(G)$ (resp, $\mc P_t(G)$) for any graph $G$ that admits $H$ as an induced subgraph. \qed
    \end{lemma}

    In light of \cref{lem:Restriction-Scarf}, induced subgraphs serve as forbidden structures of graphs for determining the Scarf property of the associated connected or path ideals. The following result about square-free monomial ideals of a certain form will play an important role in identifying such forbidden structures.

    \begin{lemma}\label{lem:scarf-2-gens}
        Let $\Bbbk$ be a field, $t$ a positive integer, and $S= \Bbbk[x_1,\dots,x_{t},x_{t+1}]$. Assume $I$ is a square-free monomial ideal of $S$ generated by monomials of degree $t$. Then $I$ is Scarf if and only if $I$ is minimally generated by at most two monomials.
    \end{lemma}

    \begin{proof}
        If $I$ is generated by at most two square-free generators of degree $t$, then it is not hard to see that its Taylor resolution is minimal. This implies that $I$ is Scarf. 
        
        Conversely, assume that $\mingens(I)=\{m_1,\dots, m_n\}$ where $n\geq 3$ and let $\sigma\subseteq \mingens(I)$ with $\lvert \sigma \rvert \geq 2$. Notice that the monomial $\lcm(\sigma)$ is of degree of at least $t+1$, and since it is square-free, we must have $\lcm(\sigma) \big\vert x_1\dots x_{t+1}$. This implies that $\lcm(\sigma)=x_1\dots x_{t+1}$, and since $n\geq 3$, the label $x_1\dots x_{t+1}$ is not unique. Thus, $\Scarf(I)$ consists only of the vertices $\{m_1,m_2,\dots,m_n\}$. Since $n \ge 3$, $\Scarf(I)$ has non-trivial reduced homology over $\Bbbk$, and thus does not support a resolution of $I$.
    \end{proof}
    
    Translating this to the context of connected and path ideals, this means that if we have a connected graph $G$ such that $\mathcal{C}_t(G)$ (resp, $\mathcal{P}_t(G)$) is Scarf for some $G$, then for any induced subgraph $H$ of $G$ with $t+1$ vertices, the ideal $\mathcal{C}_t(H)$ (resp, $\mathcal{P}_t(H)$) must have at most two minimal generators. By definition, if $G$ is a connected graph and $H$ is a subgraph with the same number of vertices, then the numbers of paths (of a fixed length) and connected induced subgraphs (of a fixed size) of $G$ are at least those of $H$. In other words, we have the following lemma.
    
    \begin{lemma}
        Let $G$ be a connected graph and $H$ a subgraph of $G$ where $\lvert V(G) \rvert = \lvert V(H) \rvert$. Then $\bigl\lvert\mingens\bigl(\mathcal{C}_t(G)\bigr)\bigr\rvert \geq \bigl\lvert\mingens\bigl(\mathcal{C}_t(H)\bigr)\bigr\rvert$ and $\bigl\lvert\mingens\bigl(\mathcal{P}_t(G)\bigr)\bigr\rvert \geq$\\ $\bigl\lvert\mingens\bigl(\mathcal{P}_t(H)\bigr)\bigr\rvert$ for any integer $t \geq 2$. \qed
    \end{lemma}
    
    Combining this with \cref{lem:Restriction-Scarf} and \cref{lem:scarf-2-gens}, we obtain a restriction lemma using subgraphs which are not necessarily induced subgraphs. 
    
    \begin{lemma}\label{thm:rest_lem_subgraph}
        Let $t\geq 2$ be an integer. Let $G$ be a connected graph and $H$ a subgraph of $G$ where $\lvert V(H) \rvert=t+1$. If $\mathcal{C}_t(G)$ (resp, $\mathcal{P}_t(G)$) is Scarf, then so is $\mathcal{C}_t(H)$ (resp, $\mathcal{P}_t(H)$). In other words, if $\mathcal{C}_t(H)$ (resp, $\mathcal{P}_t(H)$) is not Scarf, neither is $\mathcal{C}_t(G)$ (resp, $\mathcal{P}_t(G)$) for any graph $G$ that admits $H$ as a subgraph. \qed
    \end{lemma}

\section{Forbidden subgraphs}\label{sec:forbidden_subgraphs}

    The main results of this article, which are contained in \cref{sec:connected_ideals} and \cref{sec:p4_scarf}, rely primarily on the structure of the graphs involved; classifying graphs based on their (induced) subgraphs will be a repeated theme throughout. Here, we collect the graph-theoretic results needed in the later sections. 

    As above, we let $G = \bigl(V(G),E(G)\bigr)$ be a finite simple graph. For a vertex $v \in V(G)$, we will let $G - v$ denote the induced subgraph of $G$ generated by $V(G) \setminus \{v\}$. If $G$ is connected and $u,v \in V(G)$, then the \textit{distance} between $u$ and $v$ is the minimum possible length of a path connecting $u$ and $v$ in $G$. The \textit{diameter} of $G$ is the maximum distance between two vertices of $G$. If $G$ is a tree -- that is, a connected graph which contains no cycles -- then there is exactly one path between any two vertices. In this case, the diameter is simply the length of the longest path in $G$.

    Our first observation is that any connected subgraph of $G$ is contained in a connected \textit{induced} subgraph (of any given size) of $G$. 

    \begin{lemma}\label{thm:Extension-of-connected-subgraphs}
        Suppose that $G$ is a connected graph with $n$ vertices and $H$ is a connected subgraph of $G$ with $m$ vertices. For any $m \leq k \leq n$, there exists a connected induced subgraph of $G$ with $k$ vertices which contains $H$.
    \end{lemma}

    \begin{proof}
        We induct on $k$. For the base case $k = m$, we may take the induced subgraph of $G$ generated by $V(H)$. Next, suppose $m \leq k < n$ and that there exists a connected induced subgraph $H'$ of $G$ with $k$ vertices containing $H$. Since $k < n$ and $G$ is connected, there exists a vertex $v \in V(G) \setminus V(H')$ connected to a vertex of $H$ by an edge. Thus, the induced subgraph generated by $V(H') \cup \{v\}$ is connected, has $k+1$ vertices, and contains $H$.
    \end{proof}

    \begin{lemma}\label{thm:connected-after-removing-vertices}
        Let $G$ be a connected graph with $n$ vertices. For any $1 \leq k \leq n$, there are at least $\lceil n/k \rceil$ connected induced subgraphs of $G$ with $k$ vertices. In particular, if $G$ is not a single vertex, there are at least two vertices $v \in V(G)$ such that $G - v$ is connected. 
    \end{lemma}

    \begin{proof}
        Fix $1 \leq k \leq n$ and let $A_k$ denote the set of connected induced subgraphs of $G$ with $k$ vertices and suppose that $\lvert A_k \rvert = m$. \cref{thm:Extension-of-connected-subgraphs} implies that every vertex is contained in such a subgraph, since every vertex forms a connected subgraph. Thus, \(\bigl\lvert \bigcup_{H \in A_k} V(H) \bigr\rvert = \lvert V(G) \rvert = n\). On the other hand, \(\bigl\lvert\bigcup_{H \in A_k} V(H) \bigr\rvert \leq \sum_{H \in A_k} \lvert V(H) \rvert = mk\). Together, these imply that \(m \geq n/k\). 

        If \(G\) is not a single vertex, there are at least \(\bigl\lceil n/(n - 1) \bigr\rceil = 2\) connected induced subgraphs of \(G\) with \(k = n - 1\) vertices. That is, there are at least two vertices that leave a connected graph when removed from $G$.
    \end{proof}

    Let $P_n$ and $C_n$ denote the path and the cycle on $n$ vertices, respectively. We also let $S_n$ denote the star graph with $n$ leaves. Recall that the \textit{degree} of a vertex is the number of edges connected to it.

    \begin{lemma}\label{thm:cycle_path_equiv}
        The following are equivalent for a connected graph $G$: 
            \begin{enumerate}[label = (\roman*)]
                \item\label{thm:cycle_path_equiv_1} Every vertex of $G$ has degree at most two;
                \item\label{thm:cycle_path_equiv_2} $G$ does not contain $S_3$ as a subgraph;
                \item\label{thm:cycle_path_equiv_3} $G$ is a cycle or a path. 
            \end{enumerate}
        Furthermore, if every connected induced subgraph of $G$ with $k \ge 4$ vertices is a path or a cycle (with $k$ fixed), then $G$ itself is a path or a cycle.
    \end{lemma}

    \begin{proof}
        The equivalence of the first two conditions and the fact that condition \ref{thm:cycle_path_equiv_3} implies \ref{thm:cycle_path_equiv_1} are straightforward to show. So, it remains to show that \ref{thm:cycle_path_equiv_1} implies \ref{thm:cycle_path_equiv_3}. 
        We will do this by induction. If $\lvert V(G) \rvert = 1$, there is nothing to show. Suppose $\lvert V(G) \rvert = n + 1$ for some $n \geq 1$ and that the implication holds for connected graphs with $n$ vertices. By \cref{thm:connected-after-removing-vertices}, there exists $v \in V(G)$ such that $G - v$ is connected. Note that $G - v$ still satisfies \ref{thm:cycle_path_equiv_1} and thus, by induction, $G-v$ is either a path or a cycle on $n$ vertices.

        If $G-v = P_n$, then the vertex $v$ can only be connected to the leaves of $P_n$; if it is connected to one of the ``interior" vertices, the assumption that every vertex of $G$ has degree $\leq 2$ is violated. If $v$ is connected to only one leaf of $G-v$, then $G = P_{n + 1}$, and if $v$ is connected to both leaves, then $G = C_{n + 1}$. 

        Now suppose that $G - v = C_n$. The vertex $v$ cannot be an isolated vertex of $G$, since $G$ is connected, so it must share an edge with some vertex of $C_n$. But then this vertex would have degree $3$ which is not possible. So, this case cannot occur and we have that $G = P_{n + 1}$ or $G = C_{n + 1}$ as desired.

        The final statement now follows readily from the equivalence of \ref{thm:cycle_path_equiv_1}-\ref{thm:cycle_path_equiv_3}. Indeed, if $G$ is neither a path nor a cycle, then $G$ must contain $S_3$ as a subgraph. This subgraph is connected, with $\lvert V(S_3) \rvert = 4 \leq k$, so we may apply \cref{thm:Extension-of-connected-subgraphs} to conclude that there is a connected induced subgraph of $G$ with $k$ vertices containing $S_3$. Hence, this subgraph is neither a cycle nor a path, again by the equivalence of the statements. 
    \end{proof}

    \begin{lemma}\label{thm:three_verticies}
        Let $G$ be a connected graph. If $G$ is not a path, then there exist at least three distinct vertices $v \in V(G)$ such that $G - v$ is connected.
    \end{lemma}

    \begin{proof}
        We begin by making the following observation which we will use several times throughout the proof: let $v, w \in V(G)$ and suppose that $(G - v) - w$ is connected. Then $G - v$ is connected provided that $w$ shares an edge with some vertex of $G - v$, that is, a vertex of $G$ other than $v$.
        
        We proceed by induction on $\lvert V(G) \rvert$. For the base case, $\lvert V(G) \rvert = 1$, the statement is vacuously true. Assume that the result holds for connected graphs with $n \geq 1$ vertices, and let $G$ be a connected graph with $\lvert V(G) \rvert = n + 1$ which is not a path. By \cref{thm:connected-after-removing-vertices}, there exists a vertex $v \in V(G)$ such that $G' = G - v$ is connected. We consider two cases regarding the shape of $G'$.

        First, suppose that $G'$ is not a path. By the inductive hypothesis, there exist three vertices $v_1, v_2, v_3 \in V(G')$ such that $G' - v_1$, $G' - v_2$, and $G' - v_3$ are all connected. Since $G$ is connected, $v$ must share an edge with some vertex $w \in V(G')$, and at least two of $v_1, v_2, v_3$ must be different from $w$. Without loss of generality, suppose that $w \neq v_1, v_2$. By the observation above, $G - v_1$ and $G - v_2$ are connected. Thus, we have found two vertices other than $v$ that induce a connected subgraph when removed from $G$. 

        Next, suppose that $G'$ is a path. Again, since $G$ is connected, $v$ must share an edge with at least one vertex of $G'$. We consider two cases, depending on which vertices of $G'$ are adjacent to $v$: 
        
        \begin{enumerate}[label = (\roman*)]
            \item Assume $v$ is connected to an interior vertex of $G'$ and let $u$ and $w$ be the two outer vertices of the leaves of $G'$. Since $G'$ is a path, so is $G' - u$ and since $v$ shares an edge with a vertex different from $u$, we can conclude that $G-u$ is connected by the observation above. By symmetry, the same holds for $G - w$.

            \item Assume that $v$ is not connected to an interior vertex of $G'$. Notice that $v$ must be connected to both of the leaf vertices $u$ and $w$ since if it is only attached to one of them, then $G$ is a path. So, $v$ is attached to both leaves of $G'$ and therefore $G$ is a cycle. Thus, removing \textit{any} one vertex from $G$ induces a connected subgraph.
        \end{enumerate}
        
        In each case, at least three distinct vertices can be found which induce a connected subgraph when removed. This was the desired conclusion. 
    \end{proof}
    
    The next lemma classifies the trees which avoid certain forbidden subgraphs; see \cref{fig:forbidden_subgraphs}. This classification is a key ingredient to the results of \cref{sec:p4_scarf}.
    
    \begin{lemma}\label{thm:special_trees}
        The following are equivalent for a (non-empty) tree $G$.
            \begin{enumerate}[label = (\roman*)]
                \item \label{thm:special_trees_1} $G$ is isomorphic to one of the following:
                    \begin{itemize}
                        \item $S_k$ for some $k \geq 0$;
                        \item $S_k^{m,n}$ for some $k=3,4$ and $m,n \geq 0$;
                        \item $S_k^{m,n,p}$ for some $k=5,6$ and $m,n,p \geq 0$.
                    \end{itemize}
                \item \label{thm:special_trees_2} $G$ does not contain $P_9$ or the graphs $X_1, X_2, X_3,X_4$ displayed in \cref{fig:forbidden_subgraphs} as an induced subgraph.
            \end{enumerate}
    \end{lemma}
    
   \begin{figure}[H]
        \centering
        
        \begin{subfigure}[b]{0.14\textwidth}
            \centering
            \includegraphics[width=0.5\textwidth]{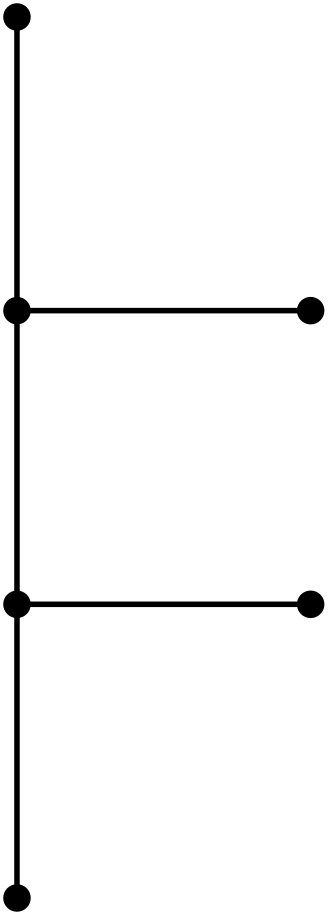}
            \caption{\(X_1\)}
        \end{subfigure}
        \hspace{7pt}
        \begin{subfigure}[b]{0.14\textwidth}
            \centering
            \includegraphics[width=\textwidth]{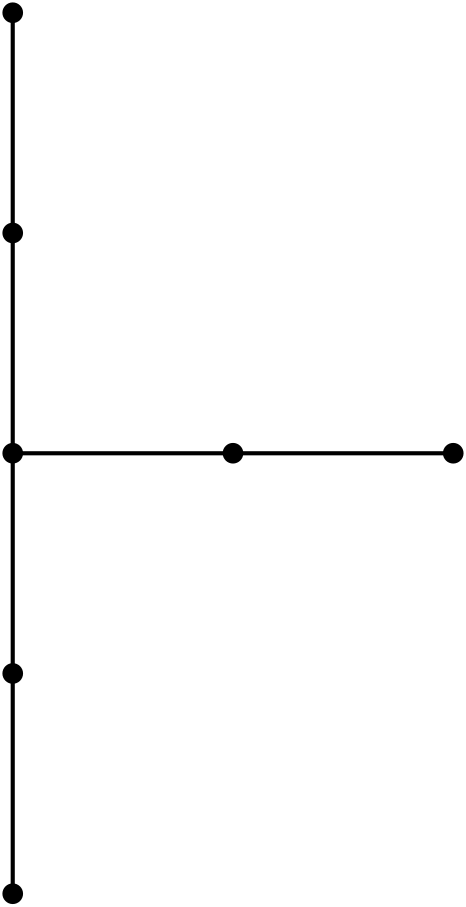}
            \caption{\(X_2\)}
        \end{subfigure}
        \hspace{7pt}
        \begin{subfigure}[b]{0.14\textwidth}
            \centering
            \includegraphics[width=.5\textwidth]{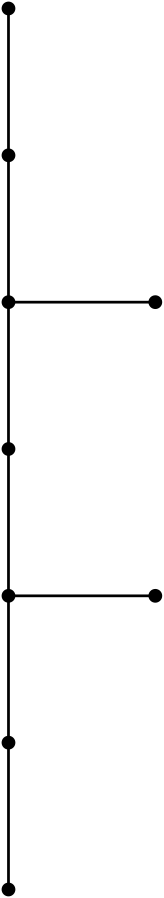}
            \caption{\(X_3\)}
        \end{subfigure}
        \hspace{7pt}
        \begin{subfigure}[b]{0.14\textwidth}
            \centering
            \includegraphics[width=.5\textwidth]{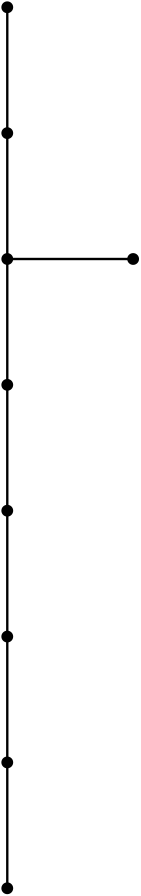}
            \caption{\(X_4\)}
        \end{subfigure}
        
        \caption{Forbidden subgraphs in \protect{\cref{thm:special_trees}}.}
        \label{fig:forbidden_subgraphs}
    \end{figure}
    
    \begin{proof}
        First, note that any connected subgraph of a tree is automatically an induced subgraph. The quantifier ``induced'' in \ref{thm:special_trees_2} is included only for emphasis as we will primarily be concerned with induced subgraphs.

        The proof that \ref{thm:special_trees_1} implies \ref{thm:special_trees_2} is a straightforward verification; see \cref{fig:S_graphs} and \cref{fig:forbidden_subgraphs}. So, we will assume that $G$ satisfies \ref{thm:special_trees_2} and show that it must have one of the forms given in \ref{thm:special_trees_1}. Note that since $G$ does not contain $P_9$, its diameter must not exceed $7$.

        If the diameter of $G$ is at most two, then $G$ must be a star graph $S_k$ for some $k \ge 0$. If the diameter of $G$ is three, then $G$ must contain $P_4$ as a subgraph.  Suppose that this path connects the vertices $v_1, v_2, v_3, v_4$, in that order. Then $G$ is built up from $P_4$ by gluing vertices of trees to one or more of the vertices $v_i$. Nothing can be glued to either $v_1$ or $v_4$, since $G$ would then have diameter greater than three in that case. If non-empty trees are glued to both $v_2$ and $v_3$, then $X_1 \subseteq G$, a contradiction. Thus, $G$ is obtained by gluing trees to one of $v_2$ or $v_3$. Furthermore, these trees must be leaves, since the diameter of $G$ would exceed $3$ otherwise. It follows that $G = S_3^{m, 0}$ for some $m \geq 1$. 

        If the diameter of $G$ is four, then $G$ contains $P_5$ as a subgraph. If this path connects the vertices $v_1,v_2,v_3,v_4,$ and $v_5$, in that order, then $G$ is built by gluing trees to some of $v_2,v_3$, and $v_4$. If non-empty trees are glued to both of $v_2, v_3$ or both of $v_3, v_4$, then $X_1 \subseteq G$. So trees can only be glued to $v_3$ or to one or both of $v_2, v_4$. In either case, the glued trees must be leaves: if anything else is glued to $v_3$, then $G$ would contain $X_2$, while gluing a non-leaf to $v_2$ or $v_4$ would cause $G$ to have diameter exceeding $4$. Thus, $G$ must be either $S_5^{0, n, 0}$ for some $n \geq 0$ or $S_3^{m, n}$ for some $m, n \geq 1$.

        The cases of diameter $5,6,$ and $7$ follow with similar reasoning. In particular, if $G$ has diameter five, then $G = S_4^{m,n}$ for some $m,n \ge 1$, or $G = S_5^{m,n,0}$ for some $m \ge 1, n \ge 0$. If $G$ has diameter six, then $G = S_5^{m,n,p}$ for some $m\ge 1, n\ge 0, p\ge 1$, or $G = S_6^{0,n,p}$ for some $n\ge 0, p\ge 1$. Finally, if $G$ has diameter seven, then $G = S_6^{m,n,p}$ for some $m,p \ge 1$ and $n \ge 0$.
    \end{proof}

\section{The Scarf complex of a connected ideal}\label{sec:connected_ideals}

    Let $G$ be a connected graph with vertices $V(G) = \{x_1,x_2,\dots,x_d\}$ and $S = \Bbbk[x_1,x_2,\dots,x_d]$ for a field $\Bbbk$. The $t$-connected ideal associated to $G$, defined in \cref{def:path_conn_intro}\ref{def:conn_intro}, is the ideal generated by all induced subgraphs of $G$ which have $t$ vertices and are connected. We denote this ideal by $\mc C_t(G)$.

    Notice that the $2$-connected ideal $\mc C_2(G)$ agrees with the standard definition of the edge ideal of $G$. The main goal of this section is to characterize, for fixed $t > 2$, the connected graphs $G$ such that $\mc C_t(G)$ is Scarf. For $t = 2$, that is, for edge ideals, such a characterization was given in \cite{faridi2024scarf}.

    \begin{theorem}[\protect{\cite[Theorem 8.3]{faridi2024scarf}}]\label{thm:faridi_edge}
        Let $G$ be a connected graph. Then the edge ideal of $G$ is Scarf if and only if $G$ is a tree of diameter not exceeding three.
    \end{theorem}

    We begin by considering the connected ideals of paths and cycles.

    \begin{proposition}\label{thm:Scarf_paths_and_cycles}
        Let $r,t\geq 3$ be integers. Then the following hold.
            \begin{enumerate}[label = (\roman*)]
                \item \label{thm:Scarf_paths_and_cycles_1} $\mc C_t(P_r)$ is Scarf if and only if $r \leq 2t$.
                \item \label{thm:Scarf_paths_and_cycles_2} $\mc C_t(C_r)$ is Scarf if and only if $r \leq t$.
            \end{enumerate}
    \end{proposition}

    \begin{proof}
        To prove the ``if" direction of ~\ref{thm:Scarf_paths_and_cycles_1}, assume $r \leq 2t$. If $r\leq t+1$, then $\mc C_t(P_r)$ is minimally generated by at most 2 generators, and therefore is Scarf. So, we may assume that $t+2\leq r \leq 2t$.

        For each $1 \leq i \leq r-t+1$, let $m_i = x_i x_{i+1} \cdots x_{i+t-1}$. Then $\mingens\bigl(\mc C_t(P_r)\bigr) = \{m_1,m_2,\dots,m_{r-t+1}\}$ is the minimal set of monomial generators for $\mc C_t(P_r)$. Since $t + 2 \leq r \leq 2t$, for any $i \leq j$, the monomials $m_i$ and $m_j$ share at least one variable, or account for \emph{all} of the variables $x_1, x_2, \dots, x_r$ (the former holds unless $r = 2t, i = 1,$ and $j = t + 1$, in which case the latter holds). It follows that $\lcm(m_i, m_j) = x_i x_{i+1} \cdots x_{j+t-1}$. This observation allows us to compute the lcm of larger subsets as well: if $\sigma$ is a non-empty subset of $\mingens\bigl(\mc C_t(P_r)\bigr)$, then $\lcm(\sigma) = x_a x_{a+1}\cdots x_b$ for some $1\leq a<b\leq r$ with $ b-a\geq t-1$. We have the following three cases:
        \begin{enumerate}[label = (\alph*)]
            \item If $b-a=t-1$, then the only monomial in $\mingens\bigl(\mc C_t(P_r)\bigr)$ that divides $\lcm(\sigma)$ is $m_a$, and this forces $\sigma=\{m_a\}$. In particular, the label of $\sigma$ is unique, that is, $\sigma \in \Scarf\bigl(\mc C_t(P_r)\bigr)$.
            \item If $b-a=t$, then the only monomials in $\mingens\bigl(\mc C_t(P_r)\bigr)$ that divide $\lcm(\sigma)$ are $m_a$ and $m_{a+1}$, and this forces $\sigma=\{m_a,m_{a+1}\}$. Again, the label of $\sigma$ is unique.
            \item If $b-a>t$, then we observe that
                \[\lcm(m_{a},m_{b-t+1}) = x_a x_{a+1}\cdots x_b = \lcm(m_{a},m_{a+1},m_{b-t+1}),\]
            and $a+1 \neq b-t+1$. So, the label $x_a x_{a+1}\cdots x_b$ is not unique and $\sigma \not\in \Scarf\bigl(\mc C_t(P_r)\bigr)$. 
        \end{enumerate}

        Putting this all together, we have that $\Delta = \Scarf\bigl(\mc C_t(P_r)\bigr)$ is a path connecting the vertices $m_1, m_2, \dots, m_{r - t +1}$. Moreover, for any monomial $m$, if two vertices $m_i$ and $m_j$ with $i < j$ divide $m$, then the vertices $m_{i+1}, m_{i+2}, \dots, m_{j-1}$ also divide $m$. It follows that $\Delta_{\leq m}$ contains each of the vertices $m_i, m_{i+1}, \dots, m_{j-1}, m_j$ and the edges connecting them, so $\Delta_{\leq m}$ is also a path. Thus, every subcomplex of $\Delta$ of the form $\Delta_{\leq m}$ is empty or acyclic over $\Bbbk$. By \cref{lem:BPS}, the ideal $\mc C_t(P_r)$ is Scarf. 
        
        For the converse, first notice that it suffices to prove that $I = \mc C_t(P_{2t+1})$ is not Scarf; the general case then follows from \cref{lem:Restriction-Scarf}. Using similar reasoning and notation as above, one can verify that the Scarf complex of $I$ is a $t+2$ sided polygon without its interior. This complex has non-trivial first homology over $\Bbbk$ and therefore $\Scarf(I)$ does not support a resolution of $I$. This completes the proof of \ref{thm:Scarf_paths_and_cycles_1}.
        
        To see that \ref{thm:Scarf_paths_and_cycles_2} holds, we first consider the case when $r=t+1$. The ideal $\mc{C}_t(C_{r}) = \mc C_t(C_{t+1})$ has $t+1$ generators and therefore is not Scarf by \cref{lem:scarf-2-gens}. For $t + 2 \leq r \leq 2t + 1$, the simplicial complex on which the Scarf complex of $\mc C_t(C_r)$ is supported is an $r$-sided polygon without interior; the verification of this is similar to \ref{thm:Scarf_paths_and_cycles_1} above. For $r > 2t + 1$, the result follows from \ref{thm:Scarf_paths_and_cycles_1} combined with \cref{lem:Restriction-Scarf}.    
    \end{proof}

    We are now ready to prove the main result of this section. 
    
    \begin{theorem}[\cref{thm:A}]\label{thm:connected-scarf-iff-short-path}
        Let $G$ be a connected graph and $t \geq 3$. Then $\mc C_t(G)$ is Scarf if and only if $G$ has at most $t$ vertices or $G = P_r$ for some $r \leq 2t$.
    \end{theorem}

    \begin{proof}
        Assume that $G$ has more than $t$ vertices and that the ideal $\mc C_t(G)$ is Scarf. Notice that, by \cref{lem:Restriction-Scarf}, $\mc C_t(G')$ is Scarf for any connected induced subgraph $G'$ of $G$. In particular, any connected induced subgraph with exactly $t+1$ vertices is Scarf and therefore must be a path by \cref{lem:scarf-2-gens} and \cref{thm:three_verticies}. It follows from \cref{thm:cycle_path_equiv} that $G$ itself is either a path or a cycle. Since $\lvert V(G) \rvert > t$ and $\mc C_t(G)$ is Scarf, we find that the only possibility given by \cref{thm:Scarf_paths_and_cycles} is $G = P_r$ for some $t + 1 \leq r \leq 2t$.

        The converse follows immediately from \cref{thm:Scarf_paths_and_cycles}. 
    \end{proof}

\section{The Scarf complex of a path ideal}\label{sec:p4_scarf}

    In this section, we consider a different but related generalization of the edge ideal of a graph, namely, the $t$-path ideal; see \cref{def:path_conn_intro}\ref{def:path_intro}. As above, we consider a connected graph $G$ with vertices $x_1,x_2,\dots,x_d$ and let $S = \Bbbk[x_1,x_2,\dots,x_d]$ where $\Bbbk$ is a field.

    Notice that $\mc P_t(G) \subseteq \mc C_t(G)$ for all $t \geq 2$ and, furthermore, $\mc P_t(G) = \mc C_t(G)$ for $t=2,3$. Thus, \cite[Theorem 8.3]{faridi2024scarf} and \cref{thm:connected-scarf-iff-short-path} characterize the graphs $G$ which have the property that $\mc P_t(G)$ is Scarf for $t = 2,3$. The goal of this section is to tackle the next smallest case, $t=4$. We will say that a graph $G$ is $\mc P_4$-Scarf if the ideal $\mc P_4(G)$ is Scarf.

    \begin{theorem}[\cref{thm:B}]\label{thm:P_4_Scarf}
        Let $G$ be a connected graph. Then $G$ is $\mc P_4$-Scarf if and only if $\lvert V(G) \rvert \leq 4$ or $G$ is isomorphic to one of the graphs listed in \cref{fig:S_graphs}:
            \begin{itemize}
                \item $T_k$ for $k \geq 0$;
                \item $S_k$ for $k \geq 0$;
                \item $S_k^{m,n}$ for $k=3,4$ and $m,n \geq 0$;
                \item $S_k^{m,n,p}$ for $k=5,6$ and $m,n,p \geq 0$.
            \end{itemize}
    \end{theorem}

    We will begin by showing that any $\mc P_4$-Scarf graph must be one of the graphs listed above. This will give us one direction of the theorem. Since each graph with four or less vertices is trivially $\mc P_4$-Scarf, we may assume $\lvert V(G) \rvert \geq 5$ throughout the rest of this section.

    \begin{proof}[Proof of \cref{thm:P_4_Scarf} (``only if'' direction)]
        Assuming that $G$ is $\mc P_4$-Scarf, we consider two cases. First, assume that $G$ is a tree. It is readily verified that none of $P_9, X_1,X_2,X_3$ and $X_4$ are $\mc P_4$-Scarf. Thus, $G$ cannot contain any of these as an induced subgraph by \cref{lem:Restriction-Scarf}. Applying \cref{thm:special_trees}, we find that $G$ must be a graph of the form $S_k$, $S_k^{m,n}$, or $S_k^{m,n,p}$. 

        Now suppose that $G$ is not a tree. We proceed by induction on the number of vertices. For the base case, $\lvert V(G) \rvert = 5$, one can verify directly that $T_2$ is the only connected non-tree graph with five vertices which is $\mc P_4$-Scarf; there are 18 connected non-tree graphs to check but the computation can be shortened significantly by using \cref{lem:scarf-2-gens} and \cref{thm:rest_lem_subgraph}. Namely, one can check that the graphs in \cref{fig:forbidden_P_4} are not $\mc P_4$-Scarf using \cref{lem:scarf-2-gens} and then apply \cref{thm:rest_lem_subgraph}.

        \begin{figure}[h]
            \centering

            \begin{subfigure}[b]{0.12\textwidth}
                \centering
                \includegraphics[width=\textwidth]{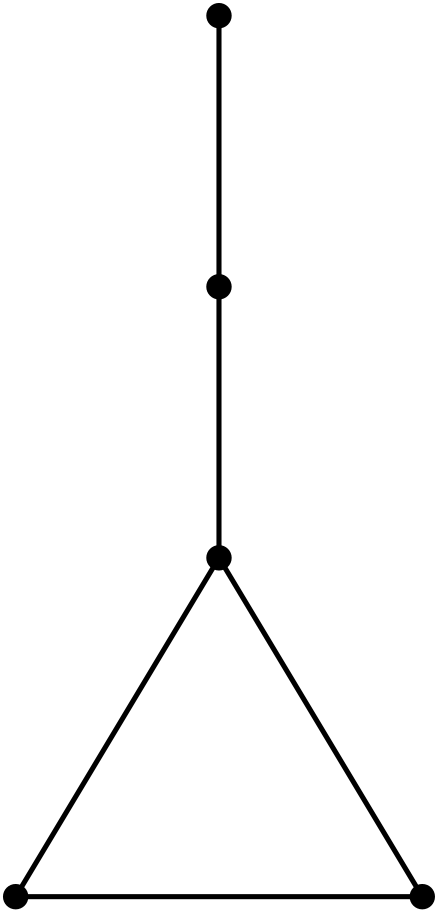}
                \caption{}
            \end{subfigure}
            \hspace{20pt}
            \begin{subfigure}[b]{0.12\textwidth}
                \centering
                \includegraphics[width=\textwidth]{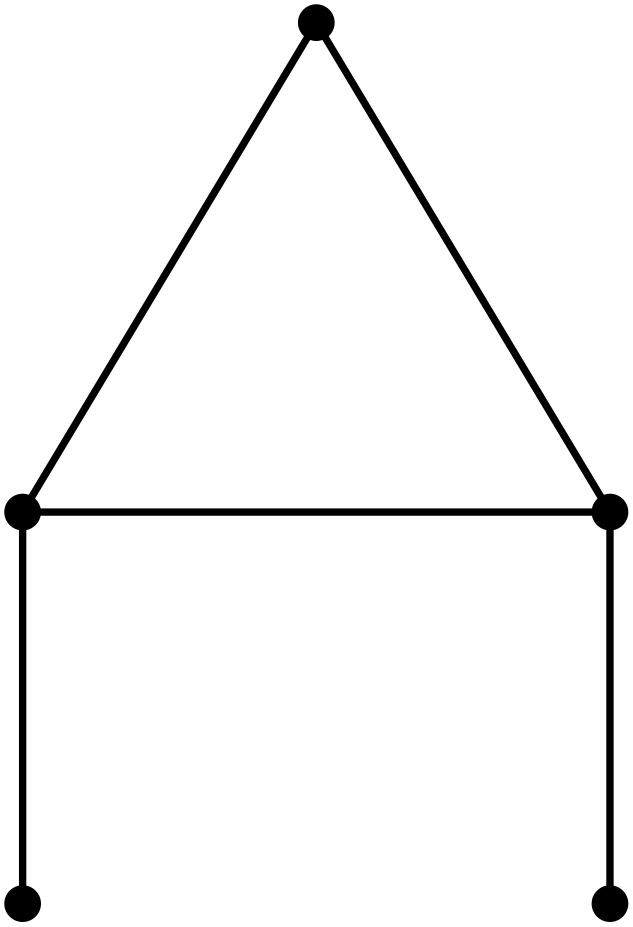}
                \caption{}
            \end{subfigure}
            \hspace{20pt}
            \begin{subfigure}[b]{0.16\textwidth}
                \centering
                \includegraphics[width=\textwidth]{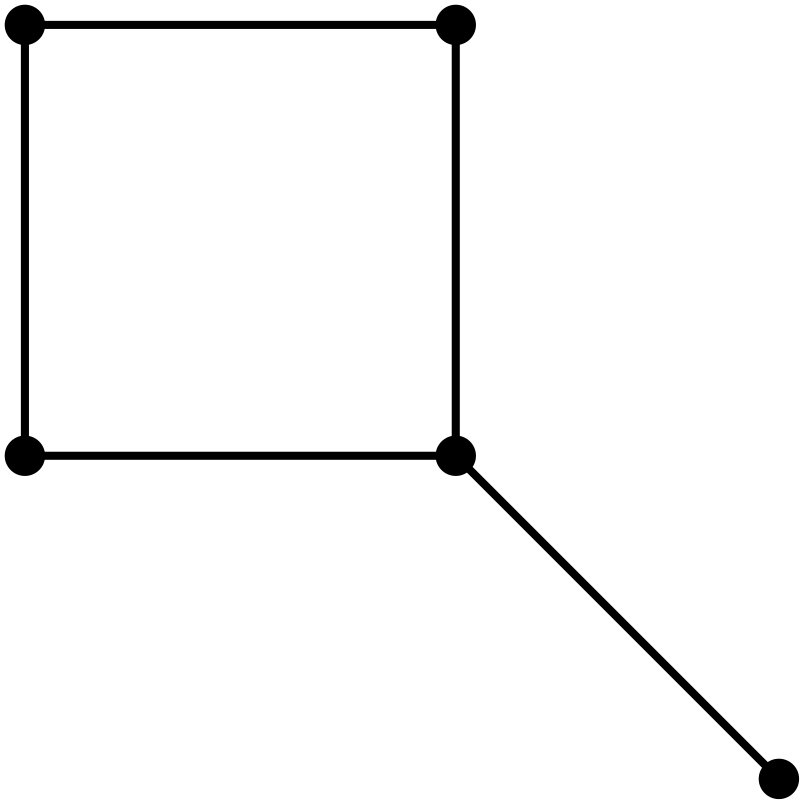}
                \caption{}
            \end{subfigure}
            \hspace{20pt}
            \begin{subfigure}[b]{0.1\textwidth}
                \centering
                \includegraphics[width=\textwidth]{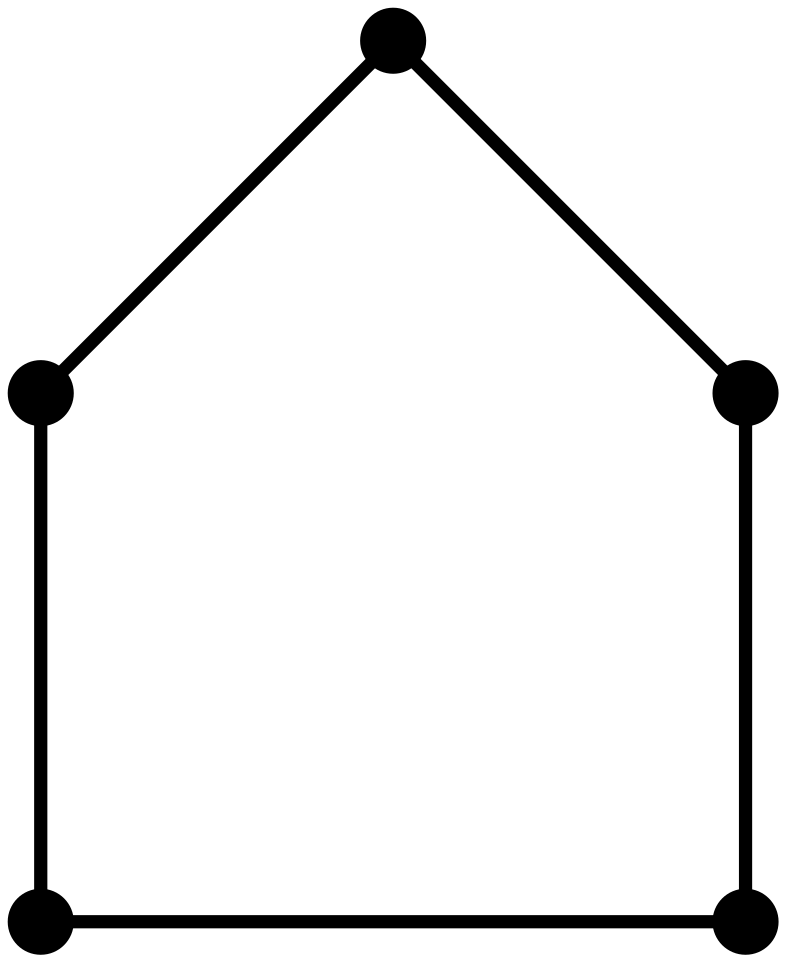}
                \caption{}
            \end{subfigure}

            \caption{Forbidden subgraphs of connected,  \(\mc P_4\)-Scarf, non-tree graphs with more than five vertices.}
            \label{fig:forbidden_P_4}
        \end{figure}
        
        Suppose now that the result is true for all connected non-tree graphs on $t$ vertices for some $t \geq 5$ and suppose $\lvert V(G) \rvert = t + 1$. Since $G$ is not a tree, $C_\ell \subseteq G$ for some $\ell \geq 3$ and we may take this $\ell$ to be minimal in the sense that $C_m \not\subseteq G$ for any $m < \ell$. It follows that $C_\ell$ must be an \emph{induced} subgraph of $G$; there can be no edge in $E(G)$ connecting non-consecutive vertices of $C_\ell$, since otherwise $C_\ell$ would contain a strictly smaller cycle. Furthermore, we must have that $\ell \leq 4$ since \cref{thm:Scarf_paths_and_cycles}\ref{thm:Scarf_paths_and_cycles_2} implies that the ideal $\mc P_4(C_\ell) = \mc C_4(C_\ell)$ is not Scarf for any $\ell \ge 5$. In fact, $\ell=4$ is not possible either since, in that case, $G$ would have to contain a subgraph of the form of ({\sc c}) in \cref{fig:forbidden_P_4}.
        But this contradicts \cref{thm:rest_lem_subgraph}. Thus, $G$ must contain $C_3$ as an induced subgraph. 
        
        By applying \cref{thm:Extension-of-connected-subgraphs}, we may extend $C_3$ to a connected induced subgraph $G - v \subseteq G$ with $t$ vertices. By induction, we have that $G-v = T_{t-3}$. The vertex $v$ must be connected to one or more of the vertices of the subgraph $T_{t-3}$. If $v$ is attached to a leaf, then removing any other leaf results in a connected non-tree graph with $t$ vertices not isomorphic to $T_{t-3}$ (such a leaf must exist since $t-3 \geq 2$). However, this is a contradiction to the inductive assumption since this subgraph must be $\mc P_4$-Scarf by \cref{lem:Restriction-Scarf}. Similarly, $v$ cannot be attached to either of the vertices of $C_3 \subseteq T_{t-3}$ without leaves. Thus, $v$ only shares an edge with the vertex of $C_3 \subseteq T_{t-3}$ to which leaves are attached, that is, $G = T_{t-2}$.
    \end{proof}

\subsection{Adding leaves}\label{sec:adding_leaves}

    To finish the proof of \cref{thm:P_4_Scarf}, we must show that all the graphs listed in the theorem are $\mc P_4$-Scarf. The key is \cref{thm:leaf_lemma}, which allows us to add leaves inductively while preserving the $\mc P_4$-Scarf property. Throughout this section, we will fix the following notation.

    \begin{notation}\label{not:adding_leaves}
        We relabel the polynomial ring $S$ in order to distinguish one variable. Let $S = \Bbbk[x,x_1,\dots,x_d]$ and $S' = S[x']$ where $x'$ is a new variable. In the application of \cref{thm:leaf_lemma}, the distinguished variable $x$ (and its counterpart, $x'$) will represent the outer vertex of a leaf; see \cref{fig:p4_notation}.
        
        Fix the following two ideals $I \subseteq S$ and $I' \subseteq S'$: 
            \begin{enumerate}
                \item Let $I \subseteq S$ be the square-free monomial ideal with minimal generating set
                    \[M = \{m_1,m_2,\dots,m_p,xn_1,xn_2,\dots,xn_q\},\]
                        where $m_i$ and $n_j$ are distinct monomials not divisible by $x$.
                \item Let $I' \subseteq S'$ be the square-free monomial ideal with minimal generating set
                    \[M' = M \cup \{x'n_1,x'n_2,\dots,x'n_q\}.\]
            \end{enumerate}
        We also let $\Gamma = \Scarf(I)$ and $\Gamma' = \Scarf(I')$ denote the Scarf complexes of $I$ and $I'$ respectively.
    \end{notation}

    \begin{figure}[h]
        \centering
        \includegraphics[width=0.35\linewidth]{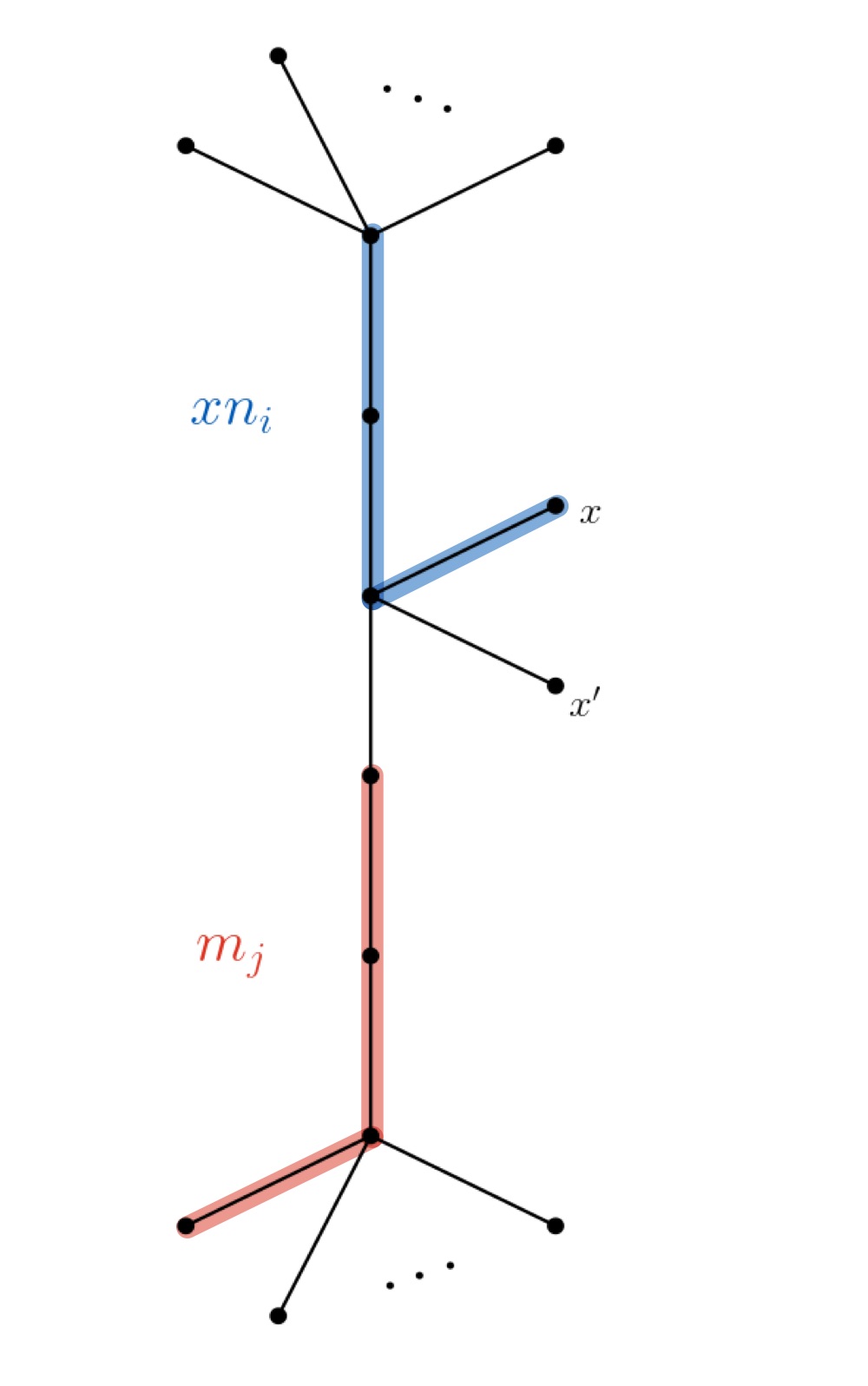}
        \caption{$S_6^{m,2,p}$ with the middle leaves distinguished.}
        \label{fig:p4_notation}
    \end{figure}

    We would like to describe $\Gamma'$ in terms of $\Gamma$. To do this, we need to relate subsets of the generating sets $M$ and subsets of $M'$. Since $M \subseteq M'$, any subset $N \subseteq M$ may also be viewed as a subset of $M'$. Conversely, given a subset $N' \subseteq M'$, we may obtain a subset of $M$ by evaluating at $x' = x$; that is, we let $\overline{N'} \subseteq M$ denote the image of $N'$ under the homomorphism $S' \to S$ which maps $x' \mapsto x$. Note that (by our assumptions in \ref{not:adding_leaves}) $N'$ and $\overline{N'}$ are comprised of square-free monomials. Thus, the lcm of either subset is the square-free product of all variables appearing in at least one of the monomials in the subset. Using this observation, we obtain the following formulas for $\lcm(N')$. 

    \begin{lemma} \label{thm:transfering_Scarf_subsets_lcms}
        The \(\lcm\) of a subset $N' \subseteq M'$ is given by
                \[\lcm(N') = \begin{cases}
                    \lcm(\overline{N'}) & \textit{ if } N' \subseteq M\\[3pt]
                    \frac{\,\, x'}{x}\lcm(\overline{N'}) & \textit{ if } x' \mid \lcm(N') \text{ and } x \nmid \lcm(N')\\[3pt]
                    x'\lcm(\overline{N'}) & \textit{ if } x,x'\mid \lcm(N')
                \end{cases}
                \]\qed
    \end{lemma} 

    \begin{lemma}\label{thm:transfering_Scarf_subsets} 
        Let $N \subseteq M$ and $N' \subseteq M'$. Then the following hold.
            \begin{enumerate}[label = (\roman*)]
                \item\label{thm:transfering_Scarf_subsets_1}$N$ is a face of $\Gamma$ if and only if $N$ is a face of $\Gamma'$.
                
                \item\label{thm:transfering_Scarf_subsets_2} If $N'$ is a face of $\Gamma'$, then $\overline{N'}$ is a face of $\Gamma$;
                \item\label{thm:transfering_Scarf_subsets_3} If $N'$ is of the form
                    \[N' = \{m_{i_1},\dots,m_{i_a},xn_{j},x'n_{j}\}\]
                for some $1 \leq i_1,\dots,i_a \leq p$, $a \ge 0$, and $1 \leq j \leq q$, then the converse of ~\ref{thm:transfering_Scarf_subsets_2} holds.
            \end{enumerate}
    \end{lemma}

    \begin{proof}
        In order to establish \ref{thm:transfering_Scarf_subsets_1}, we prove the equivalent statement that $N \not\in \Gamma$ if and only if $N \not\in \Gamma'$. First notice that if $N \not\in \Gamma$, then $N \not\in \Gamma'$ since $M \subseteq M'$. Conversely, assume that $N \not\in \Gamma'$. Then the lcm of $N$ is not unique among subsets of $M'$: that is, there exists a subset $P' \subseteq M'$ such that $P' \neq N$ but $\lcm(N) = \lcm(P')$. We claim that $P'$ must also be a subset of $M$. If not, then $x'n_i \in P'$ for some $1\leq i \leq q$. But then $x' \vert \lcm(P') = \lcm(N)$ which is not possible since $N \subseteq M$. So $P' \subseteq M$ as claimed and therefore $N \not \in \Gamma$.

        To see \ref{thm:transfering_Scarf_subsets_2}, suppose that $N \coloneqq \overline{N'} \not \in \Gamma$. We may assume that $N' \not\subseteq M$ otherwise $N=N'$ and we may apply \ref{thm:transfering_Scarf_subsets_1}. Since $N' \not\subseteq M$, we must have that $x' \vert \lcm(N')$ and therefore $x \vert \lcm(N)$. Since we are assuming that $N \not\in \Gamma$, there exists a subset $P \subseteq M$ of the form
            \[P = \{m_{i_1},\dots,m_{i_a},xn_{j_1},\dots,xn_{j_b}\}\]
        for some $a \geq 0$ and $b>0$ satisfying $P \neq N$ and $\lcm(P) = \lcm(N)$. We now consider two cases. If $x \mid \lcm(N')$, then define a subset $P'$ of $M'$ by 
        \[P' = P \cup \{x'n_{j_1}, \dots, x'n_{j_b}\}.\]
        Notice that $\overline{P'} = P$ and therefore $P' \neq N'$: otherwise, $\overline{P'} = P$ would equal $N$. Using \cref{thm:transfering_Scarf_subsets_lcms}, we have $\lcm(P') = x'\lcm(P) = x'\lcm(N) = \lcm(N')$ implying that $N' \not \in \Gamma'$. On the other hand, if $x \vert \lcm(N')$, then defining
            \[P' = \{m_{i_1}, \dots, m_{i_a}, x'n_{j_1}, \dots, x'n_{j_b}\}\]
        and reasoning as above, we find that $P' \neq N'$ and $\lcm(P') = \lcm(N')$. Hence $N' \not \in \Gamma'$ in this case as well. This completes the proof of \ref{thm:transfering_Scarf_subsets_2}.

        To establish the partial converse of ~\ref{thm:transfering_Scarf_subsets_2} given in ~\ref{thm:transfering_Scarf_subsets_3}, suppose that $N' \not \in \Gamma$ and set $N \coloneqq \overline{N'}$ as before. Then there exists $P' \subseteq M'$ such that $P' \neq N'$ and $\lcm(P') = \lcm(N')$. Notice that both $x$ and $x'$ divide the common lcm of $P'$ and $N'$ so we may apply \cref{thm:transfering_Scarf_subsets_lcms} to see that $\lcm(N) = \lcm(P)$ where $P \coloneqq \overline{P'}$. The explicit form of $N'$ assures that any differences between $P'$ and $N'$ persist when we evaluate $x' = x$. In other words, $P \neq N$, and this implies that $N \not \in \Gamma$ as desired.
    \end{proof}

    \begin{definition}
        Let $\Delta$ be a simplicial complex on a vertex set $V$. Given a face $\sigma \in \Delta$, the \textit{star neighborhood} of $\sigma$ is the subcomplex
            \[\starnbd(\sigma,\Delta) \coloneqq \{\tau \in \Delta : \sigma \cup \tau \in \Delta\}.\]
        If $v \in V$ is a vertex of $\Delta$, we will denote $\starnbd \bigl(\{v\},\Delta \bigr)$ by $\starnbd(v,\Delta)$.

        The \textit{join} of two simplicial complexes $\Delta$ and $\Delta'$, with disjoint vertex sets, is the simplicial complex consisting of all unions $\sigma \cup \sigma'$ such that $\sigma \in \Delta$ and $\sigma' \in \Delta'$. We will only consider the case when $\Delta'$ consists of a single vertex, $v$, distinct from the vertices of $\Delta$. In this case, we will refer to the join of $\Delta$ and $v$ as the \textit{cone over} $\Delta$ \textit{with vertex} $v$, and will denote it as
            \[\cone(v,\Delta) \coloneqq \Delta \cup \bigl\{\tau \cup \{v\} : \tau \in \Delta \bigr\}.\]
    \end{definition}

    \begin{lemma}\label{thm:forbidden_subsets}
        In addition to \ref{not:adding_leaves}, assume that $\starnbd(xn_i,\Gamma) \cap \starnbd(xn_j,\Gamma) = \varnothing$ for all $i\neq j$. Then the subsets $\{xn_i,xn_j\}$, $\{x'n_i,x'n_j\}$, and $\{xn_i,x'n_j\}$ are not faces of $\Gamma'$ for any $i\neq j$.
    \end{lemma}

    \begin{proof}
        Fix $1\leq i,j \leq q$ with $i\neq j$. By \cref{thm:transfering_Scarf_subsets}\ref{thm:transfering_Scarf_subsets_1}, we have $\{xn_i,xn_j\} \in \Gamma'$ if and only if $\{xn_i,xn_j\} \in \Gamma$. However, $\{xn_i,xn_j\}$ cannot be an element of $\Gamma$ since this would contradict the assumption that $\starnbd(xn_i,\Gamma) \cap \starnbd(xn_j,\Gamma) = \varnothing$. By \cref{thm:transfering_Scarf_subsets}\ref{thm:transfering_Scarf_subsets_2}, it then follows that subsets of the form $\{x'n_i,x'n_j\}$ or $\{xn_i,x'n_j\}$ cannot be faces of $\Gamma'$ either. 
    \end{proof}

    \begin{lemma}\label{thm:leaf_lemma}        
        In addition to \ref{not:adding_leaves}, assume that $\starnbd(xn_i,\Gamma) \cap \starnbd(xn_j,\Gamma) = \varnothing$ for all $i\neq j$. Then the following hold.
        \begin{enumerate}[label = (\roman*)]
            \item\label{thm:leaf_lemma_1} The simplicial complex $\Gamma'$ is obtained by replacing $\starnbd(xn_j,\Gamma)$ with \\$\cone\bigl(x'n_j, \starnbd(xn_j,\Gamma)\bigr)$ for each $1\leq j \leq q$.
            
            \item\label{thm:leaf_lemma_2} For each $1\leq j \leq q$, $\starnbd(xn_j,\Gamma') = \cone\bigl(x'n_j,\starnbd(xn_j,\Gamma)\bigr)$. In particular, $\starnbd(xn_i,\Gamma') \cap \starnbd(xn_j,\Gamma') = \varnothing$ for all $i \neq j$.
            
            \item\label{thm:leaf_lemma_3} If $I$ is Scarf, then $I'$ is Scarf.
        \end{enumerate}
    \end{lemma}

    \begin{proof}
        Since $\Gamma'$ is closed under taking subsets, any subset of $M'$ which contains one of the subsets listed in \cref{thm:forbidden_subsets} cannot be a face of $\Gamma'$. Therefore, the only subsets of $M'$ which could be faces of $\Gamma'$ are of the form
        \begin{enumerate}
            \item $\{m_{i_1},m_{i_2},\dots,m_{i_s}\}$, $s\ge 0$;
            \item $\{m_{i_1},m_{i_2},\dots,m_{i_s}, xn_j\}$, $s \ge 0, 1\leq j \leq q$;
            \item $\{m_{i_1},m_{i_2},\dots,m_{i_s}, x'n_j\}$, $s \ge 0, 1\leq j \leq q$;
            \item $\{m_{i_1},m_{i_2},\dots,m_{i_s}, xn_j, x'n_j\}$, $s \ge 0, 1\leq j \leq q$.
        \end{enumerate}
            
        Subsets of $M'$ of the form $(1)$ or $(2)$ are also subsets of $M$ and therefore are faces of $\Gamma'$ if and only if they are faces of $\Gamma$ by \cref{thm:transfering_Scarf_subsets}\ref{thm:transfering_Scarf_subsets_1}. 
        
        Let $N' = \{m_{i_1},\dots,m_{i_s},x'n_j\}$ be a face of $\Gamma'$ of the form $(3)$. Then $\overline{N'} \in \Gamma$ by \cref{thm:transfering_Scarf_subsets}\ref{thm:transfering_Scarf_subsets_2}. Thus,
            \[\overline{N'} = \{m_{i_1},\dots,m_{i_s}\} \cup \{xn_j\} \in \Gamma,\]
        which implies that $\{m_{i_1},\dots,m_{i_s}\} \in \starnbd(xn_j,\Gamma)$. Hence $N' = \tau \cup \{x'n_j\}$ for $\tau \in \starnbd(xn_j,\Gamma)$, that is, $N' \in \cone \bigl(x'n_j,\starnbd(xn_j,\Gamma)\bigr)$. Similarly, if $N'$ is of the form $(4)$, then \cref{thm:transfering_Scarf_subsets}\ref{thm:transfering_Scarf_subsets_3} implies that $N' \in \cone\bigl(x'n_j,\starnbd(xn_j,\Gamma)\bigr)$ as well.

        Thus, each face of $\Gamma'$ is either a face of $\Gamma$ or of the form $\tau \cup \{x'n_j\}$ for some face $\tau \in \starnbd(xn_j,\Gamma)$.
        Conversely, for \(\tau \in \starnbd(xn_j,\Gamma)\), any face of the form 
        \(\tau \cup \{x'n_j\}\) is in \(\Gamma'\). Together, this establishes \ref{thm:leaf_lemma_1}.

        Statement \ref{thm:leaf_lemma_2} follows from the description of \(\Gamma'\) in terms of \(\Gamma\) given in \ref{thm:leaf_lemma_1}. The equality \(\starnbd(xn_j,\Gamma') = \cone\bigl(x'n_j,\starnbd(xn_j,\Gamma)\bigr)\) is immediate. In particular, since the neighborhoods \(\starnbd(xn_j,\Gamma)\) in
        \(\Gamma\) are disjoint, the neighborhoods \(\starnbd(xn_j,\Gamma')\) in
        \(\Gamma'\) are disjoint. 

        To show the final claim \ref{thm:leaf_lemma_3}, we make the following observation: suppose $\Delta$ is a simplicial complex, $v \in \Delta$ is a vertex, and $v' \not\in \Delta$ is a ``new'' vertex distinct from those in $\Delta$. Set $C = \cone\bigl(v',\starnbd(v,\Delta)\bigr)$ and $\Delta' = \Delta \cup C$. Since $C$ and $\starnbd(v,\Delta)$ are both contractible, the Mayer–Vietoris sequence implies that $\Delta'$ and $\Delta$ have the same reduced homology over $\Bbbk$.

        Suppose that $I$ is Scarf and let $m$ be a monomial. We need to show that $\Gamma'_{\leq m}$ is acyclic or empty. There are three cases to consider:
        \begin{enumerate}[label = (\alph*)]
            \item $x' \!\!\! \not\vert m$
            \item $x' \vert m$ but $x \!\!\! \not\vert m$
            \item $x' \vert m$ and $x \vert m$.
        \end{enumerate}
        In the first case, $\Gamma'_{\leq m} = \Gamma_{\leq m}$. In the second case, the roles of $x$ and $x'$ can be interchanged to deduce that $\Gamma'_{\leq m}$ and $\Gamma_{\leq \frac{mx}{x'}}$ are isomorphic. In the third case, notice that $\Gamma'$ can be written as
            \[
                \Gamma' = \bigcup_{i=1}^q \cone\bigl(x'n_i,\starnbd(xn_i,\Gamma)\bigr) 
                \cup \Gamma.
            \]
        Assuming that both $x$ and $x'$ divide $m$, we claim that
        \begin{equation}\label{eqn:iii_equality}
            \Gamma'_{\leq m} = \bigcup_{i}
            \cone\bigl(x'n_i,\starnbd(xn_i,\Gamma_{\leq \frac{m}{\; x'}})\bigr)
            \cup \Gamma_{\leq \frac{m}{\; x'}},
        \end{equation}
        where the union is taken over all $1 \leq i \leq q$ such that $xn_i \in \Gamma_{\frac{m}{\; x'}}$. 
        
        Consider the ideals \(I^{\leq \frac{m}{\; x'}}\) and \((I')^{\leq m}\), where \(I^{\leq \frac{m}{\; x'}}\) (respectively, \((I')^{\leq m}\)) is the ideal generated by the minimal generators of $I$ which divide $m/x'$ (respectively, of $I'$ which divide $m$). These ideals satisfy the setup of \ref{not:adding_leaves} and, by definition, we have 
            \[
                \Scarf \bigl(I^{\leq \frac{m}{\; x'}} \bigr) = \Gamma_{\leq \frac{m}{\; x'}} 
                \text{ and } 
                \Scarf\bigl((I')^{\leq m}\bigr) = \Gamma'_{\leq m}.
            \] 
        Thus, we may apply \ref{thm:leaf_lemma_1} to \(I^{\leq \frac{m}{\; x'}}\) and \((I')^{\leq m}\) to obtain the equality \eqref{eqn:iii_equality}.

        Since $I$ is Scarf, $\Gamma_{\leq \frac{m}{\; x'}}$ is either acyclic or empty. Combining \eqref{eqn:iii_equality} with the Mayer-Vietoris sequence
        implies that the same is true of $\Gamma'_{\leq m}$.
    \end{proof}

    The next example illustrates the main ideas of \cref{thm:leaf_lemma} and contains the base cases for the induction needed to complete the proof of \cref{thm:P_4_Scarf}.

    \begin{example}\label{ex:base_cases}
        Consider the graph $G = S_5^{1,1,1}$ depicted below. Set $I = \mc P_4(G)$ and $\Gamma = \Scarf(I)$. Then $\Gamma$ consists of two 2-simplices connected by an edge. 

        \begin{figure}[h]
            \centering
            
            \begin{subfigure}[c]{0.49\textwidth}
                \centering
                \includegraphics[width=\textwidth]{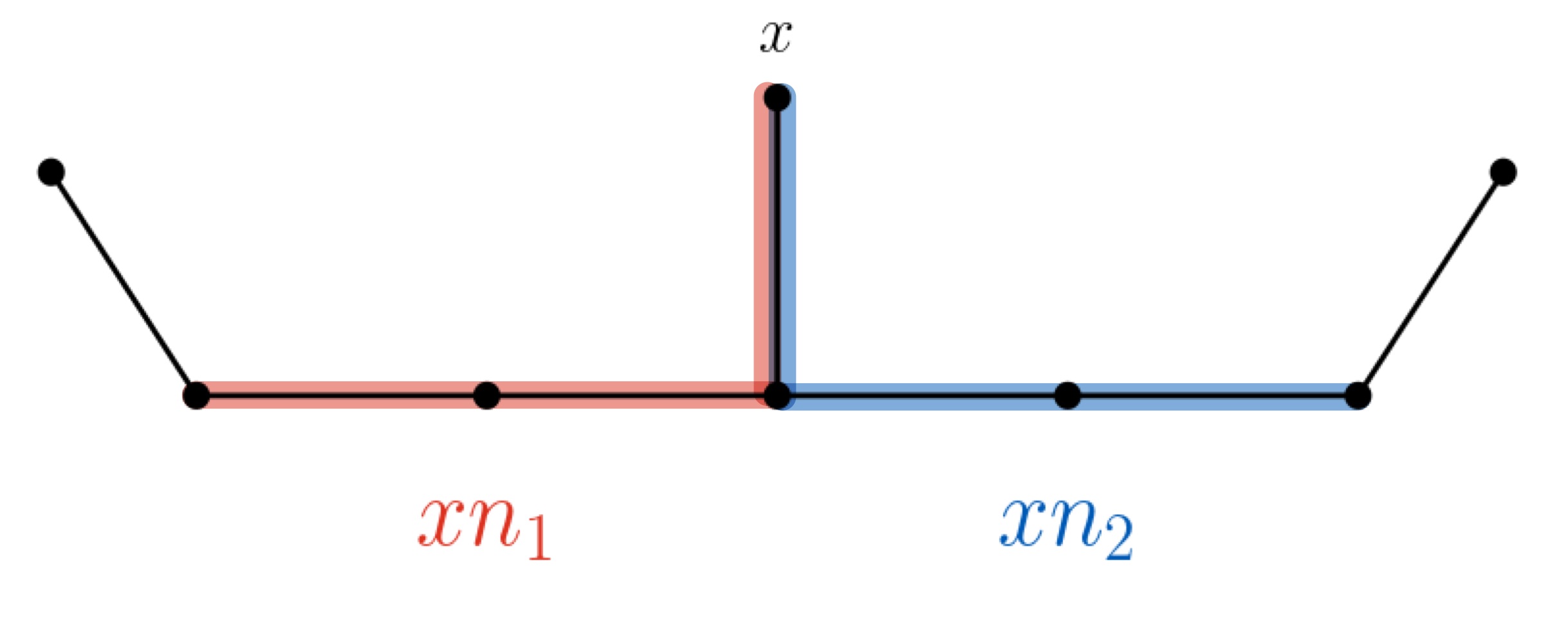}
            \end{subfigure}
            \hfill
            \begin{subfigure}[c]{0.49\textwidth}
                \centering
                \includegraphics[width=\textwidth]{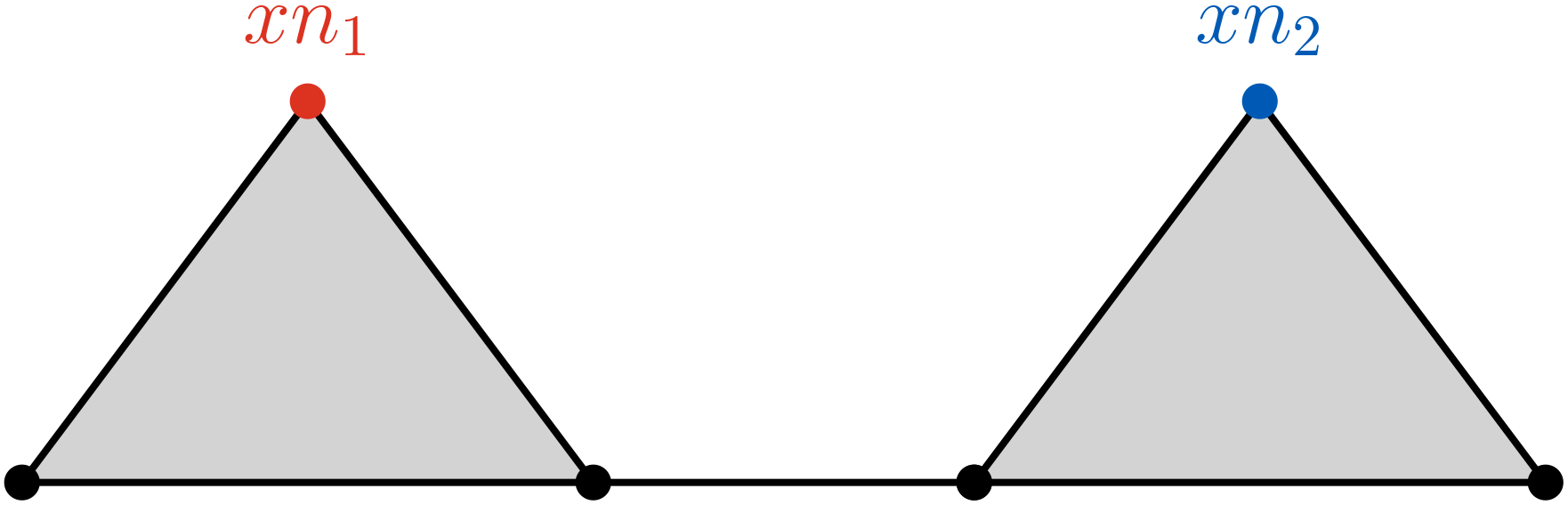}
            \end{subfigure}
            
            \caption{The graph $S_5^{1,1,1}$ and its Scarf complex.}
            \label{fig:S5_base_case}
        \end{figure}

        A straightforward computation shows that the complex $\Gamma$ has the following property: for any monomial $m$, the subcomplex $\Gamma_{\leq m}$ is either empty or contractible. Thus, $I$ is Scarf by \cref{lem:BPS}. In particular, this conclusion does not depend on the base field $\Bbbk$.

        Let $x$ denote the outer vertex of the middle leaf of $G$ as shown above. There are two paths of length four extending from $x$; using the notation of \ref{not:adding_leaves}, they are of the form $xn_1$ and $xn_2$ for some monomials $n_1,n_2$ which are not divisible by $x$. The vertices of $\Gamma$ corresponding to $xn_1$ and $xn_2$ are also labeled above in \cref{fig:S5_base_case}. Notice that $\starnbd(xn_1,\Gamma) \cap \starnbd(xn_2,\Gamma) = \varnothing$.

        To illustrate the main point of \cref{thm:leaf_lemma}, we add a new leaf at the middle vertex; this forms the graph $G' = S_5^{1,2,1}$. We label the outer vertex of the new leaf $x'$ depicted below. The resulting Scarf complex $\Gamma' = \Scarf(I')$, where $I' = \mc P_4(G')$, is also shown in \cref{fig:S_5^{1,2,1}}. It follows from \cref{thm:leaf_lemma} that $I'$ is also Scarf. Also, observe that $\starnbd(xn_1,\Gamma') \cap \starnbd(xn_2,\Gamma') = \varnothing$, as indicated in the lemma.

        \begin{figure}[h]
            \centering

            \begin{subfigure}[b]{0.49\textwidth}
                \centering
                \includegraphics[width=\textwidth]{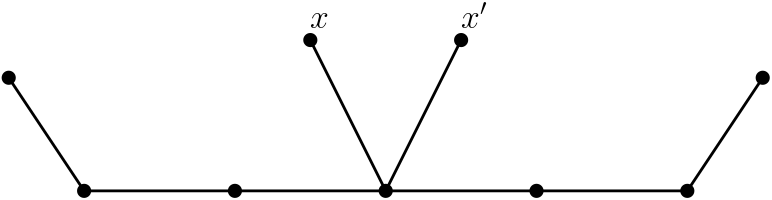}
            \end{subfigure}
            \hfill
            \begin{subfigure}[b]{0.49\textwidth}
                \centering
                \includegraphics[width=\textwidth]{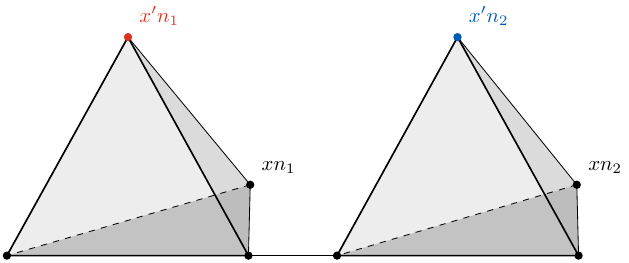}
            \end{subfigure}

            \caption{The graph $S_5^{1,2,1}$ and its Scarf complex.}
            \label{fig:S_5^{1,2,1}}
        \end{figure}

        We also consider the Scarf complex of $\mc P_4(G)$ where $G = S_6^{1,1,1}$. As in the above computation for $S_5^{1,1,1}$, the ideal $\mc P_4(G)$ is Scarf (independently of the base field $\Bbbk$) and the star neighborhoods of its Scarf complex satisfy the relevant intersection properties. \cref{thm:leaf_lemma} can be applied here as well to establish the $\mc P_4$-Scarf property of, for example, $G = S_6^{1,2,1}$; the relevant graphs and complexes are displayed in \cref{fig:S_6_base_case}.

        \begin{figure}[h]
            \centering

            \begin{subfigure}[b]{0.565\textwidth}
                \centering
                \includegraphics[width=\textwidth]{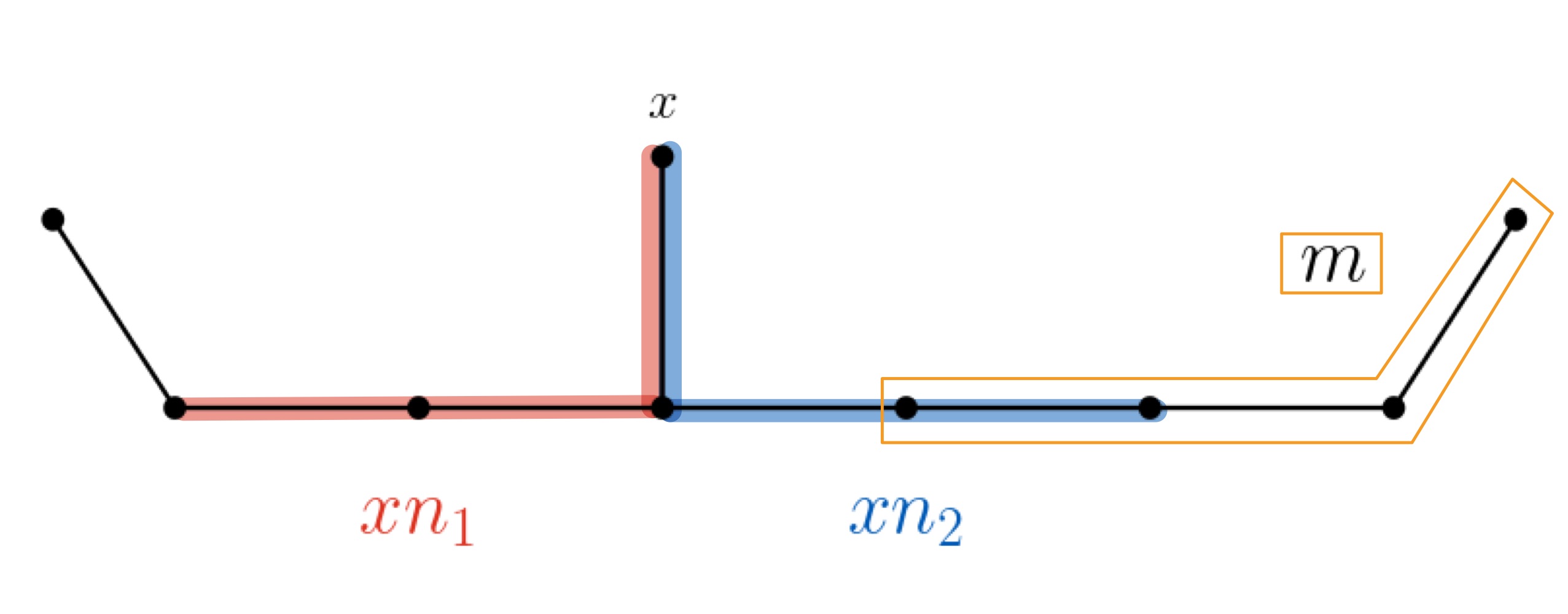}
            \end{subfigure}
            \hfill
            \begin{subfigure}[b]{0.415\textwidth}
                \centering
                \includegraphics[width=\textwidth]{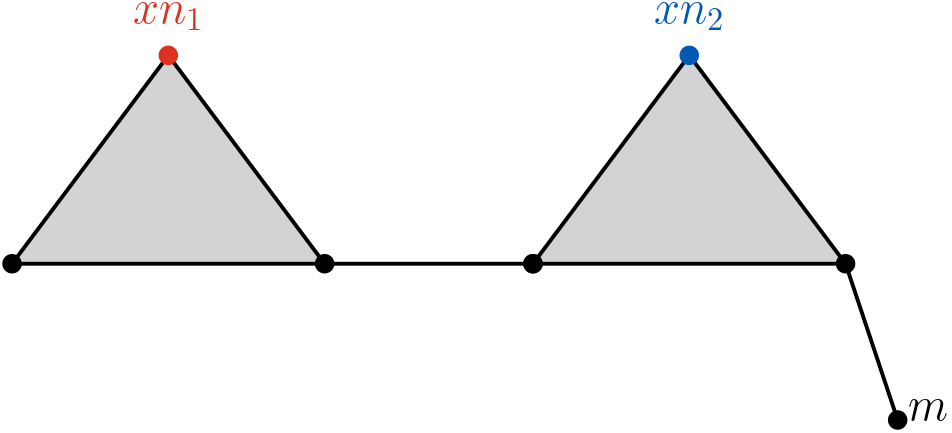}
            \end{subfigure}

            \vspace{5pt}
   
            \begin{subfigure}[c]{0.535\textwidth}
                \centering
                \includegraphics[width=\textwidth]{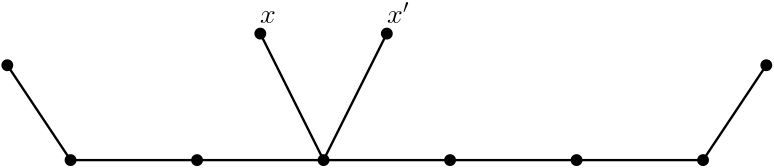}
            \end{subfigure}
            \hfill
            \begin{subfigure}[c]{0.445\textwidth}
                \centering
                \includegraphics[width=\textwidth]{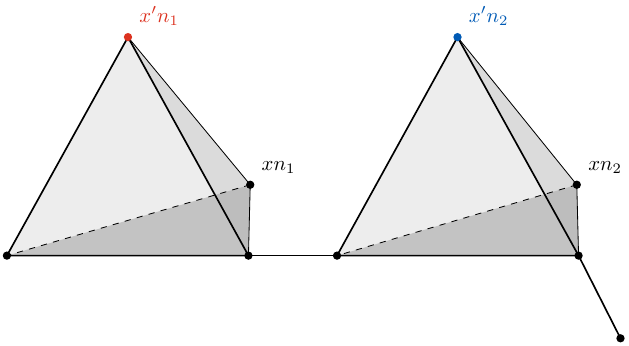}
            \end{subfigure}

                        \caption{The graphs $S_6^{1,1,1}$ and $S_6^{1,2,1}$ and their Scarf complexes.}
            \label{fig:S_6_base_case}
        \end{figure}
    \end{example}

    \begin{proof}[Proof of \cref{thm:P_4_Scarf} (``if'' direction)]
        To complete this direction of the proof, we must show that each of the graphs listed in the theorem are $\mc P_4$-Scarf. First, note that for any $k \ge 0$, the $4$-path ideal of $S_k$ is the zero ideal, and hence trivially Scarf. Also, the $4$-path ideal of $T_k$ and $S_3^{m,n}$ are minimally resolved by their Taylor resolutions and therefore are also Scarf. We also have $S_4^{m,n} \subseteq S_6^{0,m,n}$, so by \cref{lem:Restriction-Scarf}, it suffices to show that $S_k^{m,n,p}$ is Scarf for $k=5$ and $k=6$. Again by \cref{lem:Restriction-Scarf}, we in fact only need to show this for $m, n, p > 0$. 
        
        Our strategy is to consider appropriate base cases, then inductively add leaves to the specified vertices and apply \cref{thm:leaf_lemma} to show that the resulting graphs are indeed $\mc P_4$-Scarf. The base cases are $S_5^{1,1,1}$ and $S_6^{1,1,1}$, and the relevant computations for these graphs were shown above in \cref{ex:base_cases}. We now proceed by adding leaves inductively at each of the three specified vertices. First, we focus on the middle vertex. As we saw in \cref{ex:base_cases}, if we take $x$ to be the middle leaf of $S_k^{1,1,1}$ and $I = \mc P_4\bigl(S_k^{1,1,1}\bigr)$, then the disjointness hypothesis of \cref{thm:leaf_lemma} is satisfied, and hence $I' = \mc P_4\bigl(S_k^{1,2,1}\bigr)$ is Scarf. Moreover, part \ref{thm:leaf_lemma_2} of the lemma guarantees that $I'$ also satisfies the disjointness hypothesis. Continuing inductively, it follows that $S_k^{1,n,1}$ is $\mc P_4$-Scarf for all $n \ge 1$ and $k = 5,6$: given that $\mc P_4\bigl(S_k^{1,n,1}\bigr)$ is Scarf and satisfies the disjointness condition, we may apply ~\ref{thm:leaf_lemma} to see that the same is true for $\mc P_4\bigl(S_k^{1,n+1,1}\bigr)$.

        For the second induction, where we induct on $m$, the base cases are $S_k^{1,n,1}$, $k=5,6$. Here, we take $x$ to be the ``top" leaf of $S_k^{1,n,1}$ (see \cref{fig:S_graphs}). Noticing that the relevant star neighborhoods in this case are vacuously disjoint (there is only one path on four vertices containing $x$), we have that $S_k^{m,n,1}$ is $\mc P_4$-Scarf for all $m,n\ge 1$ and $k=5,6$. The third and final induction follows similarly, using $S_k^{m,n,1}$, $k=5,6$, as base cases.
    \end{proof}
    
    As a final remark, we mention that the patterns we have observed for $4$-path ideals do not seem to extend in a straightforward way for $t > 4$. For instance, for $t=5$, the subgraphs that prevent the $\mc P_5$-Scarf seem to be substantially more complicated than the ones given in \cref{fig:forbidden_subgraphs}. Indeed, each of the graphs in \cref{fig:forbidden_P_5} are not $\mc P_5$-Scarf, and we suspect that this is the complete list of obstructions for $t=5$.

    \begin{figure}[htbp!]
        \centering
        
        \begin{subfigure}[b]{0.24\textwidth}
            \centering
            \includegraphics[width=\textwidth]{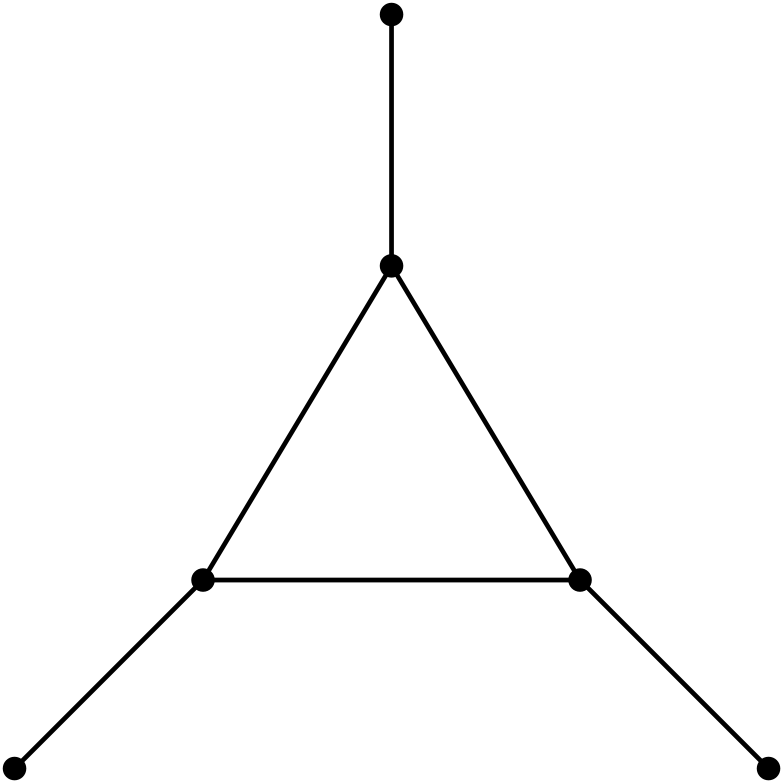}
        \end{subfigure}
        \hfill
        \begin{subfigure}[b]{0.19\textwidth}
            \centering
            \includegraphics[width=\textwidth]{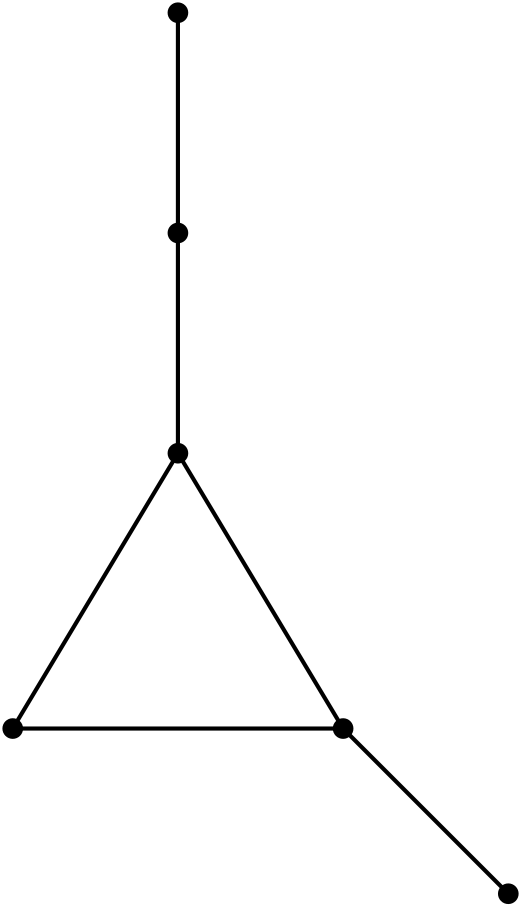}
        \end{subfigure}
        \hfill
        \begin{subfigure}[b]{0.21\textwidth}
            \centering
            \includegraphics[width=0.6\textwidth]{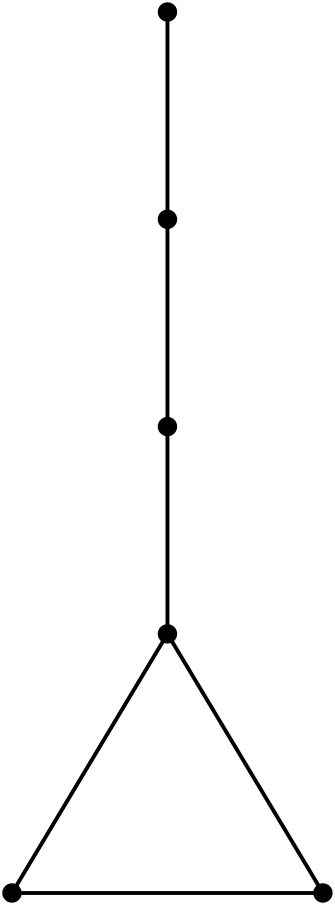}
        \end{subfigure}
        \hfill
        \begin{subfigure}[b]{0.215\textwidth}
            \centering
            \includegraphics[width=\textwidth]{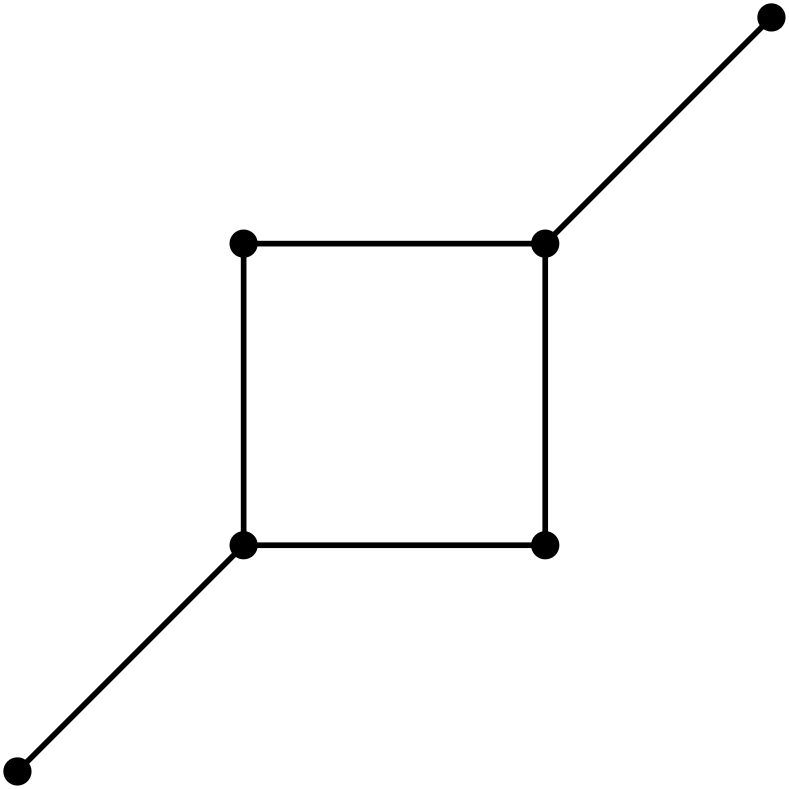}
        \end{subfigure}
        
        \vspace{10pt}
        
        \begin{subfigure}[b]{0.24\textwidth}
            \centering
            \includegraphics[width=\textwidth]{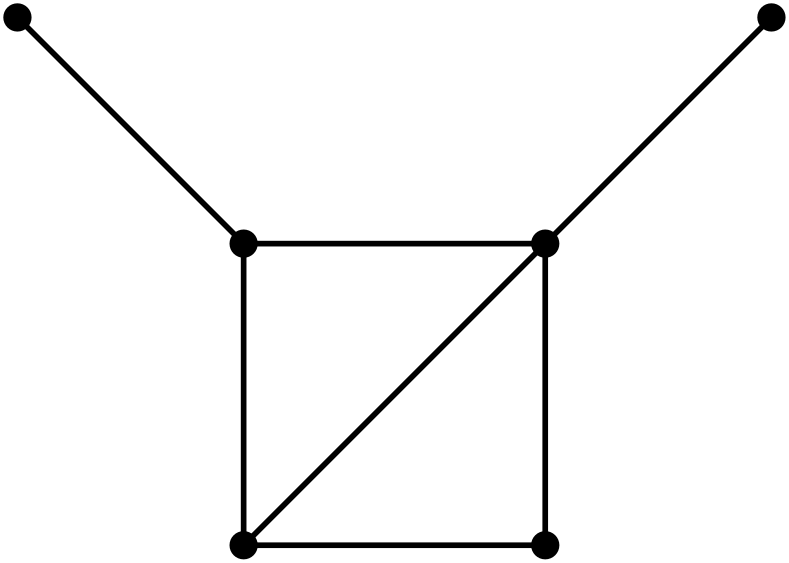}
        \end{subfigure}
        \hfill
        \begin{subfigure}[b]{0.19\textwidth}
            \centering
            \includegraphics[width=0.85\textwidth]{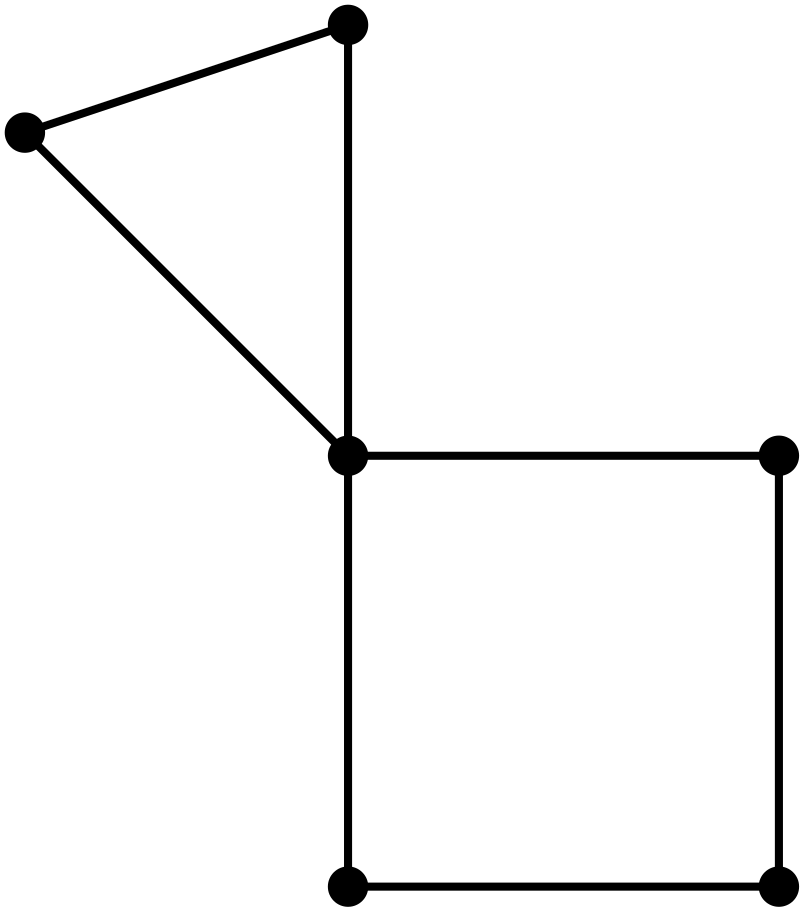}
        \end{subfigure}
        \hfill
        \begin{subfigure}[b]{0.21\textwidth}
            \centering
            \includegraphics[width=0.5\textwidth]{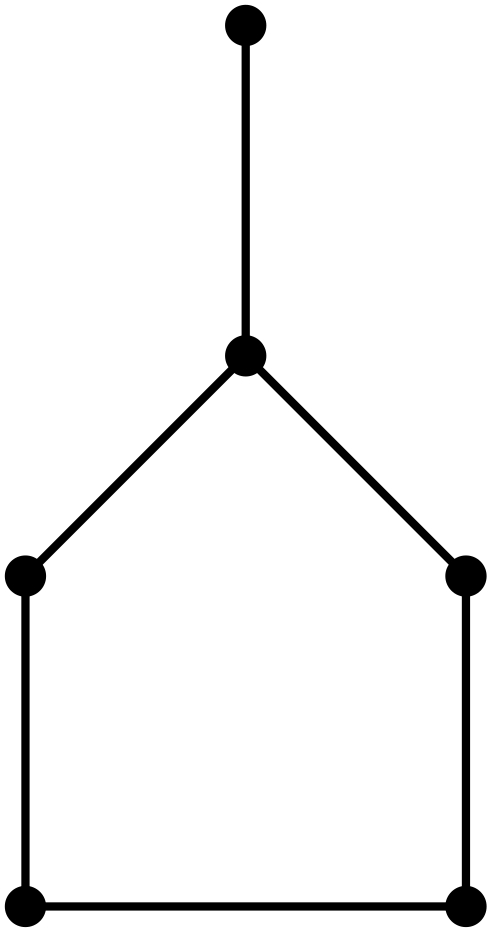}
        \end{subfigure}
        \hfill
        \begin{subfigure}[b]{0.215\textwidth}
            \centering
            \includegraphics[width=0.8\textwidth]{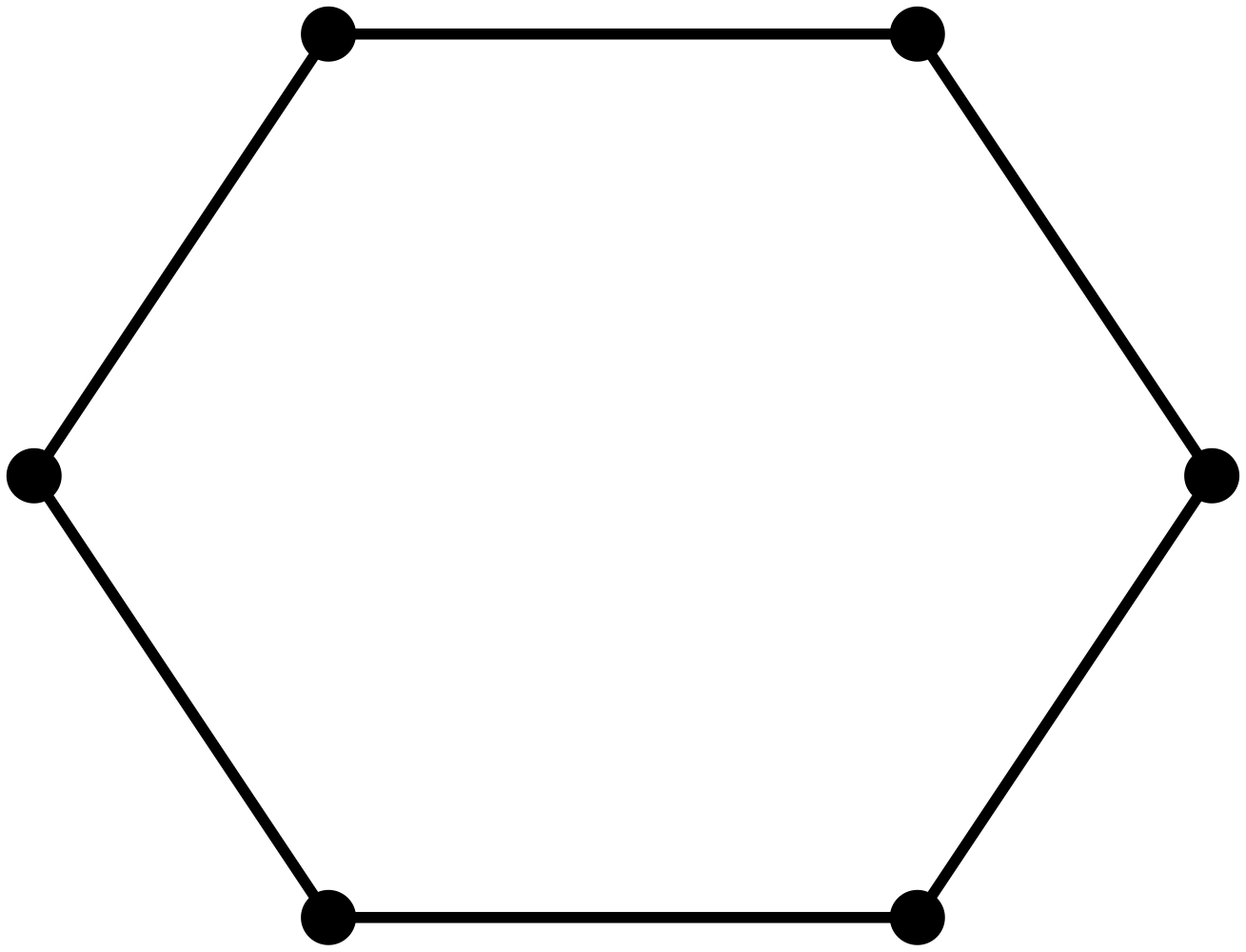}
        \end{subfigure}

        \caption{Graphs that are not \(\mc P_5\)-Scarf.}
        \label{fig:forbidden_P_5}
    \end{figure}
    
    In practice, the problem of finding forbidden structures that prevent the $\mc P_t$-Scarf property boils down to counting paths of length $t$ in a graph on $t+1$; this is due to \cref{lem:scarf-2-gens}. The difficulty of this problem is likely a symptom of its similarity to the challenging problem of finding Hamiltonian paths.

\bibliographystyle{amsalpha}
\bibliography{references}

\end{document}